\documentclass[11pt]{amsart}
\usepackage{amsmath}

\newtheorem{theorem}{Theorem}[section]
\newtheorem{definition}[theorem]{Definition}

\newtheorem{corollary}[theorem]{Corollary}

\newtheorem{proposition}[theorem]{Proposition}
\newtheorem{example}[theorem]{Example}

\newtheorem{problem}[theorem]{Problem}
\makeatletter
\newtheorem*{rep@theorem}{\rep@title}
\newcommand{\newreptheorem}[2]{%
\newenvironment{rep#1}[1]{%
 \def\rep@title{#2 \ref{##1}}%
 \begin{rep@theorem}}%
 {\end{rep@theorem}}}
\makeatother
\newreptheorem{theorem}{Theorem}
\newreptheorem{corollary}{Corollary}

\newcommand{\R}{\mathbb{R}}

\DeclareMathOperator{\e}{\epsilon}

\begin{document}

\title[Comparison Problems]{Comparison Problems for Radon Transforms}

    \author{Alexander Koldobsky, Michael Roysdon and Artem Zvavitch}

\thanks{A.K was supported in part by the National Science Foundation grant DMS-2054068.}
\thanks{M.R. was supported in part by the Zuckerman STEM Leadership program.}
\thanks{A.Z. was supported in part by the U.S. National Science Foundation Grant DMS-2000304, the United States - Israel
Binational Science Foundation (BSF) Grant 2018115, and in part by B\'ezout Labex funded by ANR, reference ANR-10-LABX-58.}

\thanks{
This material is based upon work supported by the National Science Foundation under Grant No. DMS-1929284 and the Simons Foundation under Grant No. 815891 while the authors was in residence at the Institute for Computational and Experimental Research in Mathematics in Providence, RI, during the {\it Harmonic Analysis and Convexity} semester program and while the second named author was in residence at the Institute for Computational and Experimental Research in Mathematics during the {\it Discrete Optimization: Mathematics, Algorithms, and Computation} semester program during the 2022-2023 academic year. 
}
	
\maketitle. 

\begin{abstract} Given two non-negative functions $f$ and $g$ such that the Radon transform of $f$ is pointwise smaller than the Radon transform of $g$,
does it follow that the $L^p$-norm of $f$ is smaller than the $L^p$-norm of $g$ for a given $p>0?$ We consider this problem for the classical and spherical 
Radon transforms. In both cases we point out classes of functions for which the answer is affirmative, and show that in general the answer is negative if the functions do not belong to these classes. The results are in the spirit of the solution of the Busemann-Petty problem from convex geometry, and the classes of functions that we introduce generalize the class of intersection bodies introduced by Lutwak in 1988. We also deduce slicing inequalities that are related to the well-known Oberlin-Stein type estimates for the Radon transform.
\end{abstract}

%%%%%%%%%%%%%%%%%%%%%%%%%%%%%%%%%%%%%%%%%%%%%%

\section{Introduction}

Given two non-negative functions $f,g$ such that the Radon transform of $f$ is pointwise smaller
than the Radon transform of $g$, what can one say about the functions $f$ and $g$? 
In this article, we consider this question for two kinds of Radon transforms. 

Given a function $\varphi$ on $\R^n$, integrable over all affine hyperplanes, the (classical) Radon transform of $\varphi$ is the function $\mathcal{R}\varphi$
on $\R\times S^{n-1}$ defined by 
\[
\mathcal{R}\varphi(t,\theta) = \int_{\langle x, \theta\rangle = t}\varphi(x) dx, \quad (t,\theta) \in \R \times S^{n-1},
\]
where integration is over the Lebesgue measure in the hyperplane 
perpendicular to $\theta$ at distance $t$ from the origin. 

The spherical Radon transform of a continuous function $f$ on the sphere $S^{n-1}$ is a continuous function $Rf$ on the sphere defined by
\[
Rf(\theta)= \int_{S^{n-1} \cap \theta^{\perp}} f(\xi) d\xi,\quad \theta\in S^{n-1}.
\]
Here $\theta^{\perp}$ denotes the central hyperplane orthogonal to the direction $\theta$. For the properties of the Radon transform and spherical Radon transform and different geometric applications, we refer the reader to the monographs \cite{Gardner,Groemer,Helgason,Helgason2,Kold1, Rubin1}. 
\medbreak
We consider the following comparison problems.

\begin{problem}[Comparison problem for the spherical Radon transform] \label{p:SphericalComparison} Consider two even, continuous, 
positive functions $f,g$ on $S^{n-1},\ n\ge 3$, and let $p>0$. If 
\begin{equation}\label{e:SphericalComparison}
Rf(\theta)\le Rg(\theta) \quad \text{for all } \theta \in S^{n-1},
\end{equation}
does it follow that $\|f\|_{L^p(S^{n-1})} \leq \|g\|_{L^p(S^{n-1})}?$
\end{problem}

\begin{problem}[Comparison problem for the Radon transform] \label{p:fBP} Let $p>0$. Given a pair of even, continuous functions $\varphi,\psi \colon \R^n \to \R_+,\ n\ge 2$, each of which is integrable and integrable over all affine hyperplanes, satisfying the condition:
\begin{equation} \label{e:fBPassump}
\mathcal{R}\varphi(t,\theta) \leq \mathcal{R}\psi(t,\theta), \quad \text{for all } (t,\theta)\in \R \times S^{n-1},
\end{equation}
does it follow that $\|\varphi\|_{L^p(\R^n)} \leq \|\psi\|_{L^p(\R^n)}?$
\end{problem}
By a well-known integration formula on the sphere (see for example \cite[p.28]{Kold1}) and by
Cavalieri's principle, the answers to both problems are affirmative for $p=1.$
However, for $p\neq 1$ the conclusions of the above problems maybe fail to be true in general.
For Problem~\ref{p:fBP} we show how by the following simple example.

\begin{example} 

Let $n \geq 2$, $M>1$, $c = M^{-n+1}$, and set $p > \frac{n}{n-1}.$ Consider the functions $\varphi(x) = \chi_{B_2^n}(x)$ and $\psi(x) = c \chi_{MB_2^n}(x)$.  Then it is clear that the inequality \eqref{e:fBPassump} holds, namely that 
\[
\mathcal{R}\varphi(t,\theta) \leq \mathcal{R}\psi(t,\theta), \quad \text{for all } (t,\theta) \in \R \times S^{n-1},
\]
while $\|\varphi\|_{L^p(\R^n)} > \|\psi\|_{L^p(\R^n)}.$
\end{example}

Our main source of motivation and guidance is the Busemann-Petty problem in convex geometry which was introduced in 1956 in \cite{BP} and solved at the end of 1990's. Suppose $K,L \subset \R^n$ are two origin-symmetric convex bodies so that 
\begin{equation}\label{e:BPoriginal}
|K \cap \theta^{\perp}| \leq |L \cap \theta^{\perp}|
\end{equation}
for every direction $\theta \in S^{n-1}.$ Does it necessarily follow that $|K| \leq |L|$? Here $|\cdot|$ denotes the volume of the appropriate dimension. It was proven that the answer is affirmative when $n \leq 4$ and negative when $n \ge 5$; see \cite{Gardner, Kold1} for the solution and its history. Note that Problem~\ref{p:SphericalComparison} is a generalization of the Busemann-Petty problem (choose $f = \|\cdot\|_K^{-n+1},$ $g = \|\cdot\|_L^{-n}$ and $p = \frac{n}{n-1}$). Also, it was shown in \cite{Zvavitch1,Zvavitch2} that if one considers the Busemann-Petty problem with volume replaced by an arbitrary measure with even, continuous, and positive density, then the answer remains the same. 

One of the critical ingredients in the solution of the Busemann-Petty problem is the notion of an intersection body introduced by Lutwak \cite{GLW, Lutwak}; see Section~\ref{s:prelims} below for a definition. Lutwak showed that if the body $K$ in \eqref{e:BPoriginal} is an intersection body, then the answer to the Busemann-Petty problem is affirmative. On the other hand, every origin-symmetric convex non-intersection body can be perturbed to construct a counterexample. Therefore, the answer to the Busemann-Petty problem in $\R^n$ is affirmative if, and only if every origin-symmetric convex body in $\R^n$ is an intersection body. 

Another ingredient in the Fourier analytic solution of the Busemann-Petty problem in \cite{GKS} is the characterization of intersection bodies in terms of the Fourier transform. It was proven in \cite{Kold5} that an origin-symmetric star body $K \subset \R^n$ is an intersection body if, and only if, $\|\cdot\|_K^{-1}$ represents a positive definite distribution on $\R^n$. 

Our approach to the comparison problems is based on these two ideas. We introduce special classes of functions that play the role of intersection bodies. 
For the spherical comparison problem, this is the class of functions $f$ on $S^{n-1}$ for which the extension of $f^p$ to an even homogeneous of degree -1 function on $\R^n$ represents a positive definite distribution. The results resemble Lutwak's connections in the Busemann-Petty problem.
 
\begin{theorem}\label{t:sphericalAffirmativeCase} Let $f,g$ be even continuous positive functions on the sphere $S^{n-1}$, and suppose that \begin{equation}\label{comp1}
Rf(\theta)\le Rg(\theta),\qquad \text{for all } \theta \in S^{n-1}.
\end{equation}
Then:
\begin{itemize}
\item[(a)] Suppose that for some $p>1$ the function $|x|_2^{-1}f^{p-1}\left(\frac{x}{|x|_2}\right)$ represents a positive definite distribution on $\R^n$. 
Then $\|f\|_{L^{p}(S^{n-1})} \leq \|g\|_{L^{p}(S^{n-1})}.$
\item[(b)] Suppose that for some $0 < p <1$ the function $|x|_2^{-1}g^{p-1}\left(\frac{x}{|x|_2}\right)$ represents a positive definite distribution on $\R^n$.  Then $\|f\|_{L^p(S^{n-1})} \leq \|g\|_{L^p(S^{n-1})}$.
\end{itemize}
\end{theorem}

\begin{theorem}\label{t:sphericalNegativeCase} The following hold true: 
\begin{itemize}
    \item[(a)] Let $g$ be an infinitely smooth strictly positive even function on $S^{n-1}$ and $p>1$. Suppose that the distribution $|x|_2^{-1}g^{p-1}\left(\frac{x}{|x|_2}\right)$ is not  positive definite on $\R^n$. Then there exists an infinitely smooth even function $f$ on $S^{n-1}$ so that the condition (\ref{comp1}) holds, but $\|f\|_{L^p(S^{n-1})} > \|g\|_{L^p(S^{n-1})}$.

    \item[(b)] Let $f$ be an infinitely smooth strictly positive even function on $S^{n-1}$ and $0 < p <1$. Suppose that the distribution $|x|_2^{-1}f^{p-1}\left(\frac{x}{|x|_2}\right)$ is not  positive definite on $\R^n$. Then there exists an infinitely smooth even function $g$ on $S^{n-1}$ so that the condition (\ref{comp1}) holds, but $\|f\|_{L^p(S^{n-1})} > \|g\|_{L^p(S^{n-1})}.$

\end{itemize}
\end{theorem}

To further investigate Problem~\ref{p:fBP}, we introduce the class of intersection functions.

\begin{definition} \label{d:introintfunction}
An even, continuous, non-negative, and integrable function $f$ defined on $\R^n$ is called an {\it intersection function} if, for every direction $\theta \in S^{n-1}$, the function 
\[
r \in \R \mapsto |r|^{n-1} \hat{f} (r\theta)
\]
is positive definite, where $\widehat{f}$ denotes the Fourier transforms of $f$ on $\R^n$. 
\end{definition}

We chose the Fourier definition rather than a more geometric one which is similar to the original definition of intersection bodies in \cite{GLW}. The geometric
definition now becomes a theorem, as follows.

\begin{theorem} \label{t:classification} An even, continuous, non-negative, and integrable function $f$ defined on $\R^n$ is an intersection function if, and only if, for every direction $\theta \in S^{n-1}$, there exists a non-negative, even, finite Borel measure $\mu_{\theta}$ on $\R$ such that 
\begin{itemize}
    \item the function 
\[
\theta \in S^{n-1} \mapsto \int_{\R} \mathcal{R}\varphi(t,\theta) d\mu_{\theta}(t)
\]
belongs to $L^1(S^{n-1})$ whenever $\varphi \in \mathcal{S}(\R^n)$, and 
\item 
\begin{equation}\label{e:intfunrelation}
\int_{\R^n}f \varphi = \int_{S^{n-1}} \int_{\R} \mathcal{R}\varphi(t,\theta) d\mu_{\theta}(t) d\theta.  
\end{equation}
holds for all $\varphi \in \mathcal{S}(\R^n).$
\end{itemize}
    
\end{theorem}

We will use both the Fourier definition and the geometric characterization to point out many examples of intersection functions.  In particular, we will see that the class of intersection bodies of star bodies in $\R^n$ can be identified as part of the class of intersection functions.

We now formulate analogs of Lutwak's connections for the Problem~\ref{p:fBP}.

\begin{theorem}\label{t:L2BP} Let $p >0$ and consider a pair of continuous, non-negative even functions $\varphi,\psi \in L^1(\R^n) \cap L^p(\R^n)$ satisfying the condition 
\begin{equation}\label{e:LpBPassump}
\mathcal{R}\varphi(t,\theta) \leq \mathcal{R}\psi(t,\theta) \quad \text{for all } (t,\theta) \in \R \times S^{n-1}.
\end{equation} Then:
\begin{itemize}
    \item[(a)] if $p >1$ and $\varphi^{p-1}$ is an intersection function, then $\|\varphi\|_{L^p(\R^n)} \leq \|\psi\|_{L^p(\R^n)}$, and 
    \item[(b)] if $0 < p <1$ and $\psi^{p-1}$ is an intersection function, then $\|\varphi\|_{L^p(\R^n)} \leq \|\psi\|_{L^p(\R^n)}$.
\end{itemize}

\end{theorem}

We also give a counterexample to Problem~\ref{p:fBP}. 

\begin{theorem}\label{t:RadonComparisonNegative} The following hold: 
\begin{itemize}
    \item[(a)] Fix $p >1$ and let $\psi \in \mathcal{S}(\R^n)$ be non-negative and even. If $\psi^{p-1}$ is not an intersection function, then there exists an even, non-negative  $\varphi \in \mathcal{S}(\R^n)$ such that 
\[
\mathcal{R} \varphi(t,\theta) \leq \mathcal{R} \psi(t,\theta) \quad \text{for all } (t,\theta) \in \R \times S^{n-1},
\]
but with $\|\psi\|_{L^{p}(\R^n)} < \|\varphi\|_{L^{p}(\R^n)}$.
\item[(b)] Fix $0 < p <1$ and let $\varphi \in \mathcal{S}(\R^n)$ be non-negative and even. If $\varphi^{p-1}$ is not an intersection function, then there exists a non-negative, even $\psi \in \mathcal{S}(\R^n)$ such that $\mathcal{R}\varphi \leq \mathcal{R}\psi$, but with $\|\psi\|_{L^{p}(\R^n)} < \|\varphi\|_{L^{p}(\R^n)}$.
\end{itemize} 
\end{theorem}

Since the answer to the Busemann-Petty problem is negative in most dimensions, it make sense to ask if it holds up to some absolute constant. This is the so-called {\it isomorphic Busemann-Petty problem} and was introduced in \cite{MP}: Given any pair of origin-symmetric convex bodies $K,L \subset \R^n$ satisfying the condition \eqref{e:BPoriginal}, does it follow that $|K| \leq C |L|$ for some absolute constant $C>0$? 

As shown in \cite{MP}, the isomorphic Busemann-Petty problem is equivalent to the slicing problem of Bourgain \cite{Bourgain0,Bourgain2}: Does there exist an absolute constant $C >0$ such that, for any $n \in \mathbb{N}$ and  for any origin-symmetric convex body $K$ in $\R^n$,
\[
|K|^{\frac{n-1}{n}} \leq C \max_{\theta \in S^{n-1}} |K \cap \theta^{\perp}|?
\]
Both the isomorphic Busemann-Petty problem and slicing problem remain open. In \cite{Bourgain2} Bourgain showed that $C \leq O(n^{1/4}\log(n))$. Klartag \cite{Klartag} removed logarithmic term in Bourgain's estimate. Chen  \cite{Chen} proved that $C \leq O(n^{\e})$ for every $\e>0$ as $n$ tends to infinity. Klartag and Lehec \cite{KLeh} established a polylog bound $C\leq O(\log^4n)$. The proof of Klartag and Lehec was slightly refined in \cite {JLV} to get 
$C\leq \log^{2.2226}n$. Finally, in \cite{Klartag2} Klartag improved the estimate to $C \leq \sqrt{\log n}$.

Extensions and analogs of the slicing problem to arbitrary functions was studied in \cite{CGL,GKZ,GK,KK,KL,Kold2,Kold3,Kold4,KPZ}. In particular, it was proved in \cite{Kold3} that for any $n \in \mathbb{N}$, any star body $K$ in $\R^n$ and and non-negative continuous function $f$ on $K$, one has 
\[
\int_K f \leq 2 d_{ovr}(K, \mathcal{I}_n) \max_{\theta \in S^{n-1}} \int_{K \cap \theta^{\perp}}f.
\]
Here $\mathcal{I}_n$ denotes the class of intersection bodies in $\R^n,$ and $d_{\rm ovr}(K,\mathcal{I}_n)$ is the outer volume ratio distance. In the case when $f$ is even, $K$ is origin-symmetric and convex, we may apply John's theorem \cite{John} to conclude that $d_{ovr}(K,\mathcal{I}_n) \leq \sqrt{n}$. 

An isomorphic version of the measure theoretic Busemann-Petty problem from \cite{Zvavitch2} was proved in \cite{KZ}: Given a non-negative, continuous function $f\colon \R^n \to \R_+$, and a pair of origin-symmetric convex bodies $K,L$ in $\R^n$ satisfying $\int_{K \cap \theta^{\perp}} f \leq \int_{L \cap \theta^\perp}f$ for every $\theta \in S^{n-1}$, one has 
\begin{equation} \label{e:isoBPm}
\int_K f \leq \sqrt{n} \int_L f.
\end{equation}
It is still an open problem to determine whether the constant $\sqrt{n}$ is optimal. In \cite{GKZ} the following extension of the inequality \eqref{e:isoBPm} was established: Let $K$ and $L$ be star bodies in $\R^n$ and let $f,g \colon \R^n \to \R_+$ be non-negative, continuous functions on $K$ and $L$, respectively, so that $\|g\|_{\infty} = g(0)=1$. Then 
\begin{equation}\label{e:GKZBPm}
\int_K f \leq \frac{n}{n-1} d_{ovr}(K,\mathcal{I}_n) \max_{\theta \in S^{n-1}} \left( \frac{\int_{K \cap \theta^{\perp}}f}{\int_{L \cap \theta^{\perp}} g}\right) |K|^{\frac{1}{n}} \left(\int_L g\right)^{\frac{n-1}{n}}. 
\end{equation}
For the current state of the Busemann-Petty and slicing problems for functions see the survey \cite{GKZ2}.
\medbreak
We get a slicing inequality for $p > 1$ from Theorem~\ref{t:sphericalAffirmativeCase}. In fact, if the function $g$ is constant with the value
$$g\equiv \frac{1}{|S^{n-2}|}\max_{\xi \in S^{n-1}}\int_{S^{n-1} \cap \xi^{\perp}} f(\theta)d\theta,$$
then $f$ and $g$ satisfy the conditions of Theorem \ref{t:sphericalAffirmativeCase}, and the conclusion reads as follows.
\begin{theorem}\label{t:SphericalSlicing}
Let $f$ be a positive even, continuous function on the sphere $S^{n-1}$ Assume $p > 1$ and if $|x|_2^{-1} f^{p-1}\left(\frac{x}{|x|_2}\right)$ represents a positive definite distribution on $\R^n$, then 
\begin{equation}
\|f\|_{L^{p}(S^{n-1})} \leq \frac{|S^{n-1}|^{\frac{1}{p}}}{|S^{n-2}|} \max_{\xi \in S^{n-1}} Rf(\xi).
\end{equation}

\end{theorem}

Similarly, in the case $0 < p <1$, by choosing the function
\[
f = \frac{1}{|S^{n-2}|}\min_{\theta \in S^{n-1}} \int_{S^{n-1} \cap \xi^{\perp}}g(\theta) d\theta
\]
in Theorem~\ref{t:sphericalAffirmativeCase}, we obtain:

\begin{theorem}
Let $g$ be a positive even, continuous function on the sphere $S^{n-1}$ Assume $0<p < 1$ and if $|x|_2^{-1} g^{p-1}\left(\frac{x}{|x|_2}\right)$ represents a positive definite distribution on $\R^n$, then 
\[
\|g\|_{L^{p}(S^{n-1})} \geq \frac{|S^{n-1}|^{\frac{1}{p}}}{|S^{n-2}|} \min_{\xi \in S^{n-1}} Rg(\xi).
\]
\end{theorem}

\medbreak
As proved in \cite{Kold5}, an origin-symmetric star body $K \subset \R^n$ is an intersection body if, and only if, $\|\cdot\|_K^{-1}$ represents a positive definite distribution on $\R^n$. Therefore, a positive continuous function $f$ on the sphere has the property that the distribution $f^{p-1}\cdot r^{-1}$
is positive definite if, and only if, $f=\|\cdot\|_K^{-\frac{1}{p-1}}$ for some intersection body $K.$ Combining this observation with Theorem \ref{t:SphericalSlicing},
we get that for any intersection body $K$ in $\R^n$ and any $p>1$
\begin{equation}\label{p-slicing}
\left( \int_{S^{n-1}} \|x\|_K^{-\frac{p}{p-1}} dx \right)^{\frac 1{p}} \le \frac{|S^{n-1}|^{\frac{1}{p}}}{|S^{n-2}|} 
\max_{\xi \in S^{n-1}}\left(\int_{S^{n-1} \cap \xi^{\perp}} \|x\|_K^{-\frac{1}{p-1}} dx\right).
\end{equation}
When $p=\frac n{n-1},$ the latter inequality turns into Bourgain's slicing inequality for intersection bodies. It would be interesting to see whether inequality (\ref{p-slicing}) holds
for other classes of bodies, maybe with a different constant. Note that the class of intersection bodies contains ellipsoids,  unit balls of finite-dimensional
subspaces of $L^p$ with $0<p\le 2,$ among others; see \cite[Chapter 4]{Kold1}.

We would like to point out that the inequality of Theorem \ref{t:SphericalSlicing} goes in the opposite direction to the well-known 
$L^p$-$L^q$-estimates for the Radon transform; see \cite{Christ, Christ2,Rubin} for a historical recount of such results. 
The first result of this kind was established by Oberlin and Stein in \cite{OS}: Given any function $f \colon \R^n \to \R$ belonging to $L^p(\R^n)$, one has that 
\begin{equation}\label{e:O-SRadon}
\left( \int_{S^{n-1}} \left(\int_{\R} |\mathcal{R} f(t,\theta)|^r dt\right)^{\frac{q}{r}}d\theta\right)^{\frac{1}{q}} \leq C_{n,p,q} \|f\|_{L^p(\R^n)},
\end{equation}
if, and only if, $1 \leq p < \frac{n}{n-1},$ $q \leq p'$ ($p^{-1} + p'^{-1} =1$), and $\frac{1}{r} = \frac{n}{p} - n +1$. 
In particular, inequality (\ref{e:O-SRadon}) implies that $\mathcal{R}f$ is finite almost everywhere on $\R \times S^{n-1}$ provided $f \in L^p(\R^n)$ for some $1 \leq p < \frac{n}{n-1}$.

It was also proved in \cite{OS} that for every $n \geq 3$ one has
\begin{equation}\label{e:O-SRadon2}
\left(\int_{S^{n-1}} \sup_{t \in \R}|\mathcal{R}f(t,\theta)|^s d\theta\right)^{\frac{1}{s}} \leq C_{p_1,p_2,s} \|f\|_{L^{p_1}(\R^n)}^{\alpha} \|f\|_{L^{p_2}(\R^n)}^{1-\alpha}
\end{equation}
whenever $s \leq n$, $1 \leq p_1 < \frac{n}{n-1} < p_2 \leq \infty$, and 
\[
\frac{\alpha}{p_1}+\frac{1-\alpha}{p_2} = \frac{n-1}{n}.
\]

Of particular interest to us is the limiting case of \eqref{e:O-SRadon2} due to its geometric content: If $\chi_A$ is the characteristic function of a measurable set in $A \subset \R^n$ and $s = n$, then as $p \to \frac{n}{n-1}$, inequality \eqref{e:O-SRadon2} becomes 
\begin{equation}\label{e:BusemannInt}
\left(\int_{S^{n-1}} \left(\sup_{t \in \R}|A \cap (\theta^{\perp} + t\theta)|\right)^n \right)^{\frac{1}{n}} \leq C_n |A|^{\frac{1}{n}}.
\end{equation}
If $A$ is an origin-symmetric convex body in $\R^n,$ by the Brunn concavity principle (see \cite{Gardner,Grinberg}) the supremum is achieved at $t=0,$ and one gets the Busemann intersection inequality. This connection was first observed by Lutwak in \cite{Lutwak2}.

The paper is organized as follows.  Section~\ref{s:prelims} details notations and concepts we need from harmonic analysis and convex geometry. In Section~\ref{s:sphericalRadonSolution}, we give a detailed solution to Problem~\ref{p:SphericalComparison}. In Section~\ref{s:functionoffunction}, we introduce, as an intuitive step, the notion of an intersection function of a given function, provide several examples, and prove a characterization theorem for this class of functions. Section~\ref{s:intersection functions} is dedicated to the introduction of the notion of intersection functions, those which serve as a natural extension of intersection bodies. In Section~\ref{s:Radonaffirmative} we prove Theorem~\ref{t:L2BP} and Theorem~\ref{t:RadonComparisonNegative}.

\section{Preliminaries} \label{s:prelims}

In this section we will recall several facts from harmonic analysis and convex geometry that will be used throughout the paper. 

\subsection{Notions from harmonic analysis}
We will work in the $n$-dimensional Euclidean space $\R^n$ equipped with its usual inner product structure $\langle\cdot,\cdot\rangle$ and induced normed $|\cdot|_2 = \sqrt{\langle\cdot,\cdot\rangle}$. We denote the Lebesgue measure of a measurable subset $A$ of $\R^n$ of appropriate dimension by $|A|$. The $n$-dimensional Euclidean unit ball shall be denoted by $B_2^n$, and its boundary, the unit sphere, by $S^{n-1}$.  For any fixed unit vector $\theta \in S^{n-1}$ we denote by $\theta^{\perp}$ the orthogonal complement of $\{\theta\}$, that is, $\theta^{\perp} = \{x \in \R^n \colon (x,\theta) = 0\}$. More generally, for any fixed $\theta \in S^{n-1}$ and $t\in \R$ we set 
\[
\theta^{\perp} + t \theta := \{x \in \R^n \colon \langle x,\theta\rangle = t \}
\]
to be the hyperplane parallel to $\theta^{\perp}$ at distance $t$ from the origin. We will often make use of the notation $\langle x,\theta\rangle = t$ to denote such hyperplanes in our computations below.

Given a measure metric space $(X,d,\mu)$ and $p >0$, we say that a real-valued function $h\colon X \to \R$ belongs to $L^p(X,\mu)$ if 
\[
\int_X |h(x)|^p d\mu(x) <\infty.
\]
We define the $L^p(X)$-norm of a function $h \colon X \to \R$ to be 
\[
\|h\|_{L^p(X)} = \left(\int_{X}|h(x)|^pdx \right)^{\frac{1}{p}}.
\]
We will make frequent use of the following reverse H\"older inequality \cite[pg. 135 Theorem~1]{MPF}: given a measure space $(X,\mu)$, with $\mu(X) \neq 0$, $1 < r < \infty$ and any pair of non-negative functions $h,w \in L^1(X)$, each having a positive integral, one has 
\begin{equation} \label{e:reversHolder}
\|hw\|_{L^1(X)} \geq \|h\|_{L^{\frac{1}{r}}(X)} \|w\|_{L^{-\frac{1}{r-1}}(X)}.
\end{equation}

We define the Fourier transform of a function $\varphi \in L^1(\R^n)$ by 
\[
\mathcal{F}\varphi(\xi) = \hat{\varphi}(\xi) = \int_{\R^n} \varphi(x) e^{-i(x,\xi)} dx, \quad \xi \in \R^n.
\]

For the following notions, we follow the presentation of \cite{Kold1} (see also \cite{GS, Rudin}). By $\mathcal{S}:=\mathcal{S}(\R^n)$ we denote the Schwartz space of rapidly decreasing infinitely differentiable test functions, and by $\mathcal{S}'$ the space of continuous linear functionals (distributions) acting on $\mathcal{S}$. The action of a distribution $f \in \mathcal{S'}$ on a test function $\varphi \in \mathcal{S}$ is given by integration: 
\[
\langle f, \varphi \rangle = \int_{\R^n}f(x)\varphi(x)dx. 
\]
If $\varphi$ is a test function, then so is its Fourier transform $\hat{\varphi}.$ Moreover, the Fourier transform is invertible on $S$ and its inverse is given by
\[
\Tilde{\mathcal{F}} \varphi(\xi) = (2\pi)^{-n}\int_{\R^n}\varphi(x) e^{i(x,\xi)}dx. 
\]
Consequently, for every $\varphi \in \mathcal{S}$, $(\hat{\varphi})^{\wedge}(\xi) = (2\pi)^n \varphi(-\xi)$. Also, the Fourier transform and its inverse are continuous operators on $\mathcal{S}$. By the Fubini theorem, we also have the following Parseval identity: For any pair $\varphi,\psi \in \mathcal{S}$, 
\[
\int_{\R^n} \hat{\varphi}(x) \psi(x)dx = \int_{\R^n} \varphi(\xi) \hat{\psi}(\xi)d\xi. 
\]

With this in mind, we define the Fourier transform of a distribution $f$ as a distribution $\hat{f}$ acting by $\langle \hat{f}, \varphi \rangle = \langle f, \hat{\varphi}\rangle.$

If $\varphi$ is an even test function, then 
\[
(\hat{\varphi})^{\wedge} = (2\pi)^n \varphi \quad \text{ and } \langle f, \hat{\varphi} \rangle = (2\pi)^n \langle f,\varphi\rangle.
\]

We say a distribution $f$ is a positive definite distribution if its Fourier transform is a positive distribution, i.e. for every $\varphi \in \mathcal{S}$ one has $\langle \hat{f}, \varphi) \geq 0$ whenever $\varphi \geq 0.$ We say that a complex-valued function $f$ defined on $\R^n$ is a positive definite function on $\R^n$ if, for every finite sequence $\{x_j\}_1^m$ in $\R^n$ and every choice of complex numbers $\{c_j\}_1^m$, we have 
\[
\sum_{\ell=1}^m\sum_{j=1}^mc_{\ell}\bar{c_j}f(x_{\ell}-x_{j}) \geq 0.
\]

By Bochner theorem, a function on $\R^n$ is positive definite if, and only if, it is the Fourier transform of a finite, positive Borel measure $\mu$ on $\R^n$ (see \cite{GV}). From this it follows that products of positive definite functions are again positive definite functions. 
More generally,  Schwartz's generalization of Bochner's theorem asserts that a distribution is positive definite if, and only if, if is the Fourier transform of a tempered measure on $\R^n.$

The Radon transform of a function $\varphi \in L^1(\R^n)$ that is integrable over every affine hyperplane is defined as 
\[
\mathcal{R} \varphi(t,\theta) = \int_{\theta^{\perp} + t \theta} \varphi(x) dx.
\]
Moreover, we will make frequent use of the relationship between the Fourier transform and Radon transform: Given $\varphi \in L^1(\R^n)$, for every fixed direction $\xi \in S^{n-1}$, one has that the Fourier transform of the map $t \in \R \mapsto \mathcal{R}\varphi(t,\theta)$ is equal to the function $z \in \R \mapsto \hat{\varphi}(z,\xi);$ see for example \cite[Lemma 2.11]{Kold1}.

The spherical Radon transform $R\colon C(S^{n-1}) \to C(S^{n-1})$ is a linear operator defined by 
\[
Rf(\xi) = \int_{S^{n-1} \cap \xi^{\perp}} f(x)dx, \quad \xi \in S^{n-1},
\]
for every function $f \in C(S^{n-1}).$ The spherical Radon transform is self-dual, that is, for any pair $f,g \in C(S^{n-1}),$ one has $\langle Rf,g\rangle = \langle f, Rg \rangle$.  We define the spherical Radon transform of a measure $\mu$ as a function $R\mu$ on the space $C(S^{n-1})$ acting by 
\[
\langle R\mu,f\rangle = \langle \mu, Rf\rangle = \int_{S^{n-1}} Rf(x)d\mu(x). 
\]
 
 Denote by $\mathbb{P}^n$ the space of all affine hyperplanes contained in $\R^n$. Along with the Radon transform, we consider the dual Radon transform $\mathcal{R}^*$ of an even continuous function $g \colon \mathbb{P}^n \to \R$ is defined by 
\[
(\mathcal{R}^*g)(x) = \int_{\{H \in \mathbb{P}^n\colon x \in H\}} g(H) d\nu_{n,n-1}(H),
\]
where $\nu_{n,n-1}$ denotes the rotation invariant Haar probability measure on the compact set $\{H \in \mathbb{P}^n \colon x \in H \}$.  Following \cite{Helgason} one can identify $C(\mathbb{P}^n)$ with the class even functions belonging to $C(\R \times S^{n-1})$. 

For more information on the Radon transform, see the books of Helgason \cite{Helgason,Helgason2}.

\subsection{Notions from Convexity}

We say that a compact subset $K$ of $\R^n$ is a star body if the origin $o$ belongs the the interior of $K$  and if, for every $x \in K$, each point of the interval $[o,x)$ is an interior point of $K$, and the boundary of $K$ is continuous in the sense that the Minkowski function of $K$ defined by 
\[
\|x\|_K = \min\{s \geq 0 \colon x \in sK\}
\]
is a continuous function on $\R^n$. A star body $K$ is called a convex body if in addition it is a convex set. Moreover, any star body $K$ satisfies
\[
K = \{x\in\R^n \colon \|x\|_K \leq 1\}. 
\]
A star body $K$ is said to be origin-symmetric if $K= -K.$

The radial function of a star body $K$ is defined as $\rho_K(\cdot) =\|\cdot\|_K^{-1}$; it is positive and continuous outside of the origin.  For every direction $\theta \in S^{n-1}$, $\rho_K(\theta)$ is the distance from the origin to the boundary of $K$ in the direction of $\theta$. 

Given a measure $\mu$ on $\R^n$ with non-negative continuous density $f$ and a star body $K$ in $\R^n$, we have 
\begin{equation} \label{e:polar formula measure full}
 \mu(K) = \int_{S^{n-1}} \int_0^{\|\theta\|_K^{-1}} r^{n-1}f(r\theta)dr\theta,
\end{equation}
which, in the case of the volume, becomes 
\[
|K| = \frac{1}{n}\int_{S^{n-1}} \|\theta\|_K^{-n}d\theta. 
\]

Given any hyperplane $\xi^{\perp}$, the polar formula for measure of the section $K \cap \xi^{\perp}$ is given by 
\begin{equation} \label{e:polar formula measure}
\begin{split}
\mu(K \cap \xi^{\perp}) &= \int_{S^{n-1} \cap \xi^{\perp}} \int_0^{\|\theta\|_K^{-1}} r^{n-2}f(r\theta)drd\theta \\
&= R\left( \int_{0}^{\|\cdot\|_K^{-1}} r^{n-2} f(r\cdot) dr\right)(\xi),  
\end{split}
\end{equation}
and for the volume: 
\[
|K \cap \xi^{\perp}| = \frac{1}{n-1} R\left( \| \cdot\|_K^{-n+1}\right)(\xi).
\]

If $f$ is a continuous function on $S^{n-1}$ and $0<p<n,$ we denote by $f\cdot r^{-p}$ the extension of $f$ to
an even homogeneous function of degree $-p$ on $\R^n:$
$$f\cdot r^{-p}(x)= |x|_2^{-p} f\left(\frac x{|x|_2}\right) \quad x\in \R^n\setminus \{0\}.$$
Since $0<p<n,$ this function is locally integrable on $\R^n$ and represents a distribution. 
Suppose that $f$ is infinitely smooth, i.e. $f\in C^{\infty}(S^{n-1})$. Then by \cite[Lemma 3.16]{Kold1}, the Fourier transform
in the sense of distributions
$$(f\cdot r^{-p})^\wedge = g\cdot r^{-n+p},$$
for some function $g\in C^{\infty}(S^{n-1})$. When we write $(f\cdot r^{-p})^\wedge(\xi),$ we mean $g(\xi),\ \xi \in S^{n-1}.$
If $f,g$ are infinitely smooth functions on $S^{n-1}$, we have the following spherical version of Parseval's
formula (see \cite[Lemma 3.22]{Kold1}):  for any $p\in (-n,0)$
\begin{equation}\label{parseval}
\int_{S^{n-1}} (f\cdot r^{-p})^\wedge(\xi) (g\cdot r^{-n+p})^\wedge(\xi) =
(2\pi)^n \int_{S^{n-1}} f(\theta) g(\theta)\ d\theta.
\end{equation}
Suppose $f$ is a continuous function on $S^{n-1}$. The Fourier transform of $f\cdot r^{-n+1}$ is a continuous function on the sphere. More precisely, by \cite[Lemma 3.7]{Kold1},
\begin{equation}\label{section-fourier}
(f\cdot r^{-n+1})^\wedge= \pi Rf\cdot r^{-1}.
\end{equation}
We will also use a non-smooth version of Parseval's formula from \cite[Corollary 3.23]{Kold1}.
If $g\in C(S^{n-1})$ and $g\cdot r^{-1}$ is a positive definite distribution, then there exists a finite Borel measure
$\mu_g$ on $S^{n-1}$ so that for every $f\in C(S^{n-1})$
\begin{equation}\label{parseval1}
\int_{S^{n-1}} (f\cdot r^{-n+1})^\wedge(\xi) d\mu_g(\xi) =
(2\pi)^n \int_{S^{n-1}} f(\theta) g(\theta) d\theta.
\end{equation}

Note that (\ref{parseval}) and (\ref{parseval1}) are formulated in \cite{Kold1} specifically for Minkowski functionals, not functions. However,
the Minkowski functional of a star body is an arbitrary continuous positive function on $S^{n-1}.$

The class of intersection bodies of star bodies was introduced by Lutwak in \cite{Lutwak} during his investigation of the Busemann-Petty problem.
We say that a star body $K$ is an {\it intersection body of a star body} $L$ if, for every direction $\theta \in S^{n-1}$, one has 
\[
\|\xi\|_K^{-1} = |L \cap \xi^{\perp}| = \frac{1}{n-1} R\left( \| \cdot\|_L^{-n+1}\right)(\xi).
\]

Following \cite{GLW}, we say that a star body $K$ is an {\it intersection body} if there exists a finite, positive Borel measure $\mu$ on the sphere $S^{n-1}$ so that $\|\cdot\|_K^{-1} = R\mu$ as functionals on $C(S^{n-1})$; that is, for every continuous functions $f$ on $S^{n-1}$, 
\begin{equation} \label{e: intersection body general}
\int_{S^{n-1}} \|\theta\|_K^{-1}f(\theta)d\theta = \int_{S^{n-1}} Rf(x) d\mu(x).     
\end{equation}
We denote by $\mathcal{I}_n$ the class of intersection bodies in $\R^n$; it is immediately clear that $\mathcal{I}_n$ contains the class of intersection bodies of star bodies. 

In \cite{Kold5} it was proved that an origin-symmetric star body $K$ in $\R^n$ is an intersection body if, and only if, $\|\cdot\|_K^{-1}$ is a positive definite distribution on $\R^n$.  This result was used in the solution of the Busemann-Petty problem in \cite{GKS}.

\section{The spherical case}\label{s:sphericalRadonSolution}

In this section prove Theorems~\ref{t:sphericalAffirmativeCase} and \ref{t:sphericalNegativeCase}. Begin by recalling the statement of Theorem~\ref{t:sphericalAffirmativeCase}:
\begin{reptheorem}{t:sphericalAffirmativeCase} Let $f,g$ be even continuous positive functions on the sphere $S^{n-1}$, and suppose that
\[Rf(\theta)\le Rg(\theta),\qquad \text{for all } \theta \in S^{n-1}.
\]
Then:
\begin{itemize}
\item[(a)] Suppose that for some $p>1$ the function $|x|_2^{-1}f^{p-1}\left(\frac{x}{|x|_2}\right)$ represents a positive definite distribution on $\R^n$. 
Then $\|f\|_{L^{p}(S^{n-1})} \leq \|g\|_{L^{p}(S^{n-1})}.$
\item[(b)] Suppose that for some $0 < p <1$ the function $|x|_2^{-1}g^{p-1}\left(\frac{x}{|x|_2}\right)$ represents a positive definite distribution on $\R^n$.  Then $\|f\|_{L^p(S^{n-1})} \leq \|g\|_{L^p(S^{n-1})}$.
\end{itemize}
\end{reptheorem}

\begin{proof} We begin with the proof of part (a). By (\ref{section-fourier}), the inequality $Rf\le Rg$ can be written as
\begin{equation}\label{e:FourierSphere}
(f\cdot r^{-n+1})^\wedge (\theta)\le (g\cdot r^{-n+1})^\wedge (\theta), \qquad \forall \theta\in S^{n-1}.
\end{equation}
Integrating both sides of the latter inequality over $S^{n-1}$ with respect to the non-negative measure $\mu_{f^{p-1}}$ corresponding to the 
positive definite distribution $f^{p-1}\cdot r^{-1}$ by (\ref{parseval1}), we get
$$\int_{S^{n-1}} (f\cdot r^{-n+1})^\wedge(\theta) d\mu_{f^{p-1}}(\theta) \le
\int_{S^{n-1}} (g\cdot r^{-n+1})^\wedge(\theta) d\mu_{f^{p-1}}(\theta).$$
Applying  Parseval's identity (\ref{parseval1}) we get
$$\int_{S^{n-1}} f^{p}(\theta)\ d\theta \le \int_{S^{n-1}} f^{p-1}(\theta) g(\theta)\ d\theta, 
$$
and using H\"older's inequality we get
$$\int_{S^{n-1}} f^{p}(\theta)\ d\theta \le \int_{S^{n-1}} f^{p-1}(\theta) g(\theta)\ d\theta \le \left(\int_{S^{n-1}} f^{p}(\theta) d\theta \right)^{\frac{p-1}{p}}
\left(\int_{S^{n-1}} g^{p}(\theta) d\theta \right)^{\frac 1{p}},$$
and the result follows, which completes the proof of part (a). 

The proof of (b) is more or less identical to the proof of part (a), except with the use of H\"older's inequality below replaced with its reverse \eqref{e:reversHolder} for $0 <p <1$.

Integrating both sides of the inequality \eqref{e:FourierSphere} over $S^{n-1}$ with respect to the non-negative measure $\mu_{g^{p-1}}$ corresponding to the 
positive definite distribution $g^{p-1}\cdot r^{-1}$ by (\ref{parseval1}), and
applying  Parseval's identity (\ref{parseval1}) we get
$$\int_{S^{n-1}} f(\theta)g^{p-1}(\theta) d\theta \le \int_{S^{n-1}} g^p(\theta)\ d\theta, 
$$
and using the reverse H\"older's inequality \eqref{e:reversHolder} with $X = S^{n-1}$, $d\mu = d\theta$, $h=f$, $w=g^{p-1}$ and $r=1/p$, we get
$$\int_{S^{n-1}} g^{p}(\theta)\ d\theta \geq \int_{S^{n-1}} f(\theta) g^{p-1}(\theta)\ d\theta \geq \left(\int_{S^{n-1}} f^{p}(\theta) d\theta \right)^{\frac{1}{p}}
\left(\int_{S^{n-1}} g^{p}(\theta) d\theta \right)^{\frac{p-1}{p}},$$
and the result follows, which completes the proof of part (b).

\end{proof}

\medbreak

To conclude this section, we provide a proof of Theorem~\ref{t:sphericalNegativeCase}. We restate it here for convenience of the reader. 

\begin{reptheorem}{t:sphericalNegativeCase} The following hold true: 
\begin{itemize}
    \item[(a)] Let $g$ be an infinitely smooth strictly positive even function on $S^{n-1}$ and $p>1$. Suppose that the distribution $|x|_2^{-1}g^{p-1}\left(\frac{x}{|x|_2}\right)$ is not  positive definite on $\R^n$. Then there exists an infinitely smooth even function $f$ on $S^{n-1}$ so that the condition (\ref{comp1}) holds, but $\|f\|_{L^p(S^{n-1})} > \|g\|_{L^p(S^{n-1})}$.

    \item[(b)] Let $f$ be an infinitely smooth strictly positive even function on $S^{n-1}$ and $0 < p <1$. Suppose that the distribution $|x|_2^{-1}f^{p-1}\left(\frac{x}{|x|_2}\right)$ is not  positive definite on $\R^n$. Then there exists an infinitely smooth even function $g$ on $S^{n-1}$ so that the condition (\ref{comp1}) holds, but $\|f\|_{L^p(S^{n-1})} > \|g\|_{L^p(S^{n-1})}.$
\end{itemize}
\end{reptheorem}

\begin{proof} We begin with the proof of part (a). 
Since $g$ is infinitely differentiable and positive, the Fourier transform of $g^{p-1}\cdot r^{-1}$ is of the form $h\cdot r^{-n+1}$ where $h$ is an even infinitely differentiable function on the sphere; see \cite[Lemma 3.16]{Kold1}. This function is negative on some open symmetric set $\Omega$. Choose a function $\psi\in C^\infty(S^{n-1})$ so that $\psi\ge 0$ everywhere on $S^{n-1}$ and $\psi>0$ only on some non-empty open subset of $\Omega.$ The Fourier transform of $\psi\cdot r^{-1}$ is a function $\varphi \cdot r^{-n+1}$ where $\varphi\in C^\infty(S^{n-1})$, again by \cite[Lemma 3.16]{Kold1}. Then $(\varphi \cdot r^{-n+1})^\wedge= (2\pi)^n \psi \cdot r^{-1}.$

Define the function $f$ on $S^{n-1}$ by
$$f(\theta)= g(\theta) - \e \varphi(\theta),\qquad \forall \theta \in S^{n-1},$$
where $\e$ is small enough so that $f>0$ on the sphere. By extending the functions $f,g,\varphi$ to $\R^n$ homogeneously of degree $-n+1$, we then have
$$(f\cdot r^{-n+1})^\wedge = (g\cdot r^{-n+1})^\wedge - (2\pi)^n \psi\cdot r^{-1}.$$
Since $\psi$ is non-negative everywhere on the sphere, by (\ref{section-fourier}) we conclude that the functions $f$ and $g$ satisfy the condition (\ref{comp1}).
Multiplying the latter equality by $(g^{p-1}\cdot r^{-1})^\wedge$ and integrating over the sphere we get
$$\int_{S^{n-1}} (f\cdot r^{-n+1})^\wedge(\theta) (g^{p-1}\cdot r^{-1})^\wedge(\theta)\ d\theta$$ 
$$=\int_{S^{n-1}} (g\cdot r^{-n+1})^\wedge(\theta) (g^{p-1}\cdot r^{-1})^\wedge(\theta)\ d\theta - \e(2\pi)^n\int_{S^{n-1}}  \psi(\theta)(g^{p-1}\cdot r^{-1})^\wedge(\theta)\ d\theta.$$
Parseval's formula (\ref{parseval})  implies that
$$\int_{S^{n-1}} f(\theta)g^{p-1}(\theta)\ d\theta = \int_{S^{n-1}} g^{p}(\theta)\ d\theta - \e
(2\pi)^n \int_{S^{n-1}} \psi(\theta)(g^{p-1}\cdot r^{-1})^\wedge(\theta)\ d\theta.$$
Since $\psi$ can be positive only where $(g^{p-1}\cdot r^{-1})^\wedge$ is negative, we get 
$$\int_{S^{n-1}} g^{p}(\theta)\ d\theta < \int_{S^{n-1}} f(\theta)g^{p-1}(\theta)\ d\theta,$$
and by H\"older's inequality
$$\int_{S^{n-1}} g^{p}(\theta)\ d\theta < \int_{S^{n-1}} f^{p}(\theta)\ d\theta,$$ which completes the proof of part (a).

Now we move to the proof of (b). Since $f$ is infinitely differentiable and positive, the Fourier transform of $f^{p-1}\cdot r^{-1}$ is of the form $h\cdot r^{-n+1}$ where $h$ is an even infinitely differentiable function on the sphere; see \cite[Lemma 3.16]{Kold1}. This function is negative on some open symmetric set $\Omega$ in the sphere. Choose a function $\psi\in C^\infty(S^{n-1})$ so that $\psi\ge 0$ everywhere on $S^{n-1}$ and $\psi>0$ only on some non-empty open subset of $\Omega.$ The Fourier transform of $\psi\cdot r^{-1}$ is a function $\varphi \cdot r^{-n+1}$ where $\varphi\in C^\infty(S^{n-1})$, again by \cite[Lemma 3.16]{Kold1}. Then $(\varphi \cdot r^{-n+1})^\wedge= (2\pi)^n \psi \cdot r^{-1}.$

Define the function $g$ on $S^{n-1}$ by
$$g(\theta)= f(\theta) + \e \varphi(\theta),\qquad \forall \theta \in S^{n-1},$$
where $\e$ is small enough so that $g>0$ on the sphere. By extending the functions $f,g,\varphi$ to $\R^n$ homogeneously of degree $-n+1$, we then have
$$(g\cdot r^{-n+1})^\wedge = (f\cdot r^{-n+1})^\wedge + (2\pi)^n \psi\cdot r^{-1}.$$
Since $\psi$ is non-negative everywhere on the sphere, by (\ref{section-fourier}) we conclude that the functions $f$ and $g$ satisfy the condition (\ref{comp1}).
Multiplying the latter equality by $(f^{p-1}\cdot r^{-1})^\wedge$ and integrating over the sphere we get
$$\int_{S^{n-1}} (g\cdot r^{-n+1})^\wedge(\theta) (f^{p-1}\cdot r^{-1})^\wedge(\theta)\ d\theta$$ 
$$=\int_{S^{n-1}} (f\cdot r^{-n+1})^\wedge(\theta) (f^{p-1}\cdot r^{-1})^\wedge(\theta)\ d\theta + \e(2\pi)^n\int_{S^{n-1}}  \psi(\theta)(f^{p-1}\cdot r^{-1})^\wedge(\theta)\ d\theta.$$
Parseval's formula (\ref{parseval})  implies that
$$\int_{S^{n-1}} g(\theta)f^{p-1}(\theta)\ d\theta = \int_{S^{n-1}} f^{p}(\theta)\ d\theta + \e
(2\pi)^n \int_{S^{n-1}} \psi(\theta)(f^{p-1}\cdot r^{-1})^\wedge(\theta)\ d\theta.$$
Since $\psi$ can be positive only where $(f^{p-1}\cdot r^{-1})^\wedge$ is negative, we get 
$$ \int_{S^{n-1}} g(\theta)f^{p-1}(\theta)\ d\theta < \int_{S^{n-1}} f^{p}(\theta)\ d\theta $$
and by the reverse H\"older's inequality \eqref{e:reversHolder} we get
$$\int_{S^{n-1}} g^{p}(\theta)\ d\theta < \int_{S^{n-1}} f^{p}(\theta)\ d\theta.$$
 \end{proof}
 
 Note that the condition of Theorem \ref{t:sphericalNegativeCase} that $g$ is infinitely smooth can be removed using the approximation argument
 of \cite[Lemma 4.10]{Kold1}, but then $g$ needs to be perturbed twice to construct a counterexample.
 
\section{The intersection function of a function} \label{s:functionoffunction}

The concept of an intersection body in two steps \cite{GLW,Lutwak}. First, he gave a geometrically clear definition of an intersection body of
a star body, and then replaced star bodies by measures to define the general concept of an intersection body. We do it in a similar way for intersection functions. First, we introduce the intersection function of a positive function.

Denote by $\mathcal{L}^n$ the class of positive, continuous, integrable, and even in the first variable  functions  on $\R\times S^{n-1}.$

\begin{definition} \label{d:intfunctionoffunctions} Given $g\in \mathcal{L}^n,$  
we say that a function $f$ on $\R^n$ is an intersection function of $g$ if, for any Schwartz test function $\varphi \in \mathcal{S}(\R^n)$,
\begin{equation}\label{non-central}
 \int_{\R^n} \varphi(x) f(x)\ dx = \int_{S^{n-1}}\int_{\R} \mathcal{R}\varphi(t,\theta) g(t,\theta)\ dt\ d\theta.
\end{equation} 
\end{definition}
Essentially, this means that $f=R^{*}g$ is the dual Radon transform of a positive function $g;$ see \cite[p.3]{Helgason2}. 
The existence of an intersection function is guaranteed by the well-known formula for the dual Radon transform.

\begin{proposition}\label{t:radon} Let $g\in \mathcal{L}^n,$ then the function $f \colon \R^n \to \R_+$ defined by 
\[
f(x) = \int_{S^{n-1}} g(\langle x,\theta\rangle,\theta) d\theta
\]
is an intersection function of $g$.
\end{proposition}

\begin{proof} By Fubini's theorem, we have that 
\begin{align*}
\langle f, \varphi\rangle &= \int_{\R^n}f(x)\varphi(x)\ dx \\&=\int_{\R^n}\left(\int_{S^{n-1}}g(\langle x,\theta\rangle,\theta )d\theta \right) \varphi(x)\ dx\\
&= \int_{S^{n-1}} \int_{\R^n} \varphi(x) g(\langle x,\theta\rangle),\theta)\ dx\ d\theta\\
&= \int_{S^{n-1}} \int_{\R}\left(\int_{\theta^{\perp}+t\theta}\varphi(x) dx \right)g(t,\theta)\ dt\ d\theta\\
&= \int_{S^{n-1}} \int_{\R} \mathcal{R}\varphi(t,\theta) g(t,\theta)\ dt\ d\theta,
\end{align*}
whenever $\varphi$ is a Schwartz test function on $\R^n$.  This means that the function $f$, as defined above, satisfies the condition \eqref{non-central} for any Schwartz test function on $\R^n$, and so it must be an intersection function of the function $g$. 
\end{proof}

This simple formula is not very effective if one wants to know weather a given function $f$ is an intersection function of a function.
As it was done in the case of intersection bodies, we establish the Fourier characterization of intersection functions which works better for
our purposes. 

\begin{proposition} \label{t:fourier} Let $g\in \mathcal{L}^n.$ A function $f$ on $\R^n$ is an intersection function of $g$ if, and only if,
$$f =\frac 1{\pi}\left( |x|_2^{-n+1}\left(g\left(t,\frac{x}{|x|_2}\right)\right)^\wedge_t(|x|_2)\right)^\wedge_x,$$
where the interior Fourier transform is taken with respect to $t\in \R,$ and the exterior Fourier
transform is with respect to $x\in \R^n$. 
\end{proposition}

\begin{proof}
Note that for fixed $\theta \in S^{n-1}$ the function $t \in \R\to \mathcal{R} \hat{\varphi}(t,\theta)$ is the Fourier transform
of the function $z \in \R \to (2\pi)^{n-1}{\varphi}(z\theta).$ Therefore, for any test function $\varphi,$ applying Parseval to the inner integral by $dt$, we get
\begin{align*}
\langle \hat{f},\varphi \rangle &= \int_{\R^n} f(x)\hat{\varphi}(x)\ dx= \int_{S^{n-1}}\int_{\R} \mathcal{R}\hat{\varphi}(t,\theta) g(t,\theta)\ dt\ d\theta\\
&=(2\pi)^{n-1}\int_{S^{n-1}}\int_{\R} \varphi(z\theta) (g(t,\theta))^\wedge_t(z)\ dz\ d\theta\\
&=2(2\pi)^{n-1}\Big\langle |x|_2^{-n+1}\left(g\left(t,\frac{x}{|x|_2}\right)\right)^\wedge_t(|x|_2), \varphi(x)\Big\rangle.
\end{align*}
\end{proof}

\bigbreak

 From the Propositions \ref{t:radon} and \ref{t:fourier}, respectively, we have the following corollary.

\begin{corollary}\label{t:fourier-radon} For any $g\in \mathcal{L}^n$, one has 
$$\frac 1{\pi}\left( |x|_2^{-n+1}\left(g\left(t,\frac{x}{|x|_2}\right)\right)^\wedge_t(|x|_2)\right)^\wedge_x(\xi)=
\int_{S^{n-1}} g\left(\langle\xi,\theta\rangle,\theta\right) \ d\theta,\quad \forall \xi\in \R^n.$$
Moreover, for every $r \in \R$ and $\theta \in S^{n-1}$, the following identity holds: 
\begin{equation}\label{e:relation}
 (g(t,\theta))_t^\wedge(r)= |r|^{n-1}\hat{f}(r\theta),
\end{equation}
where $f$ is the intersection function of $g$.
\end{corollary}

We also have the following uniqueness theorem:

\begin{corollary} Given a pair $g_1,g_2 \in \mathcal{L}^n$ such that 
\[
\int_{S^{n-1}}g_1(\langle x,\theta \rangle, \theta) d\theta = \int_{S^{n-1}}g_2(\langle x,\theta \rangle, \theta) d\theta \quad \text{for all } x\in \R^n,
\]
one has that $g_1=g_2$.
    
\end{corollary}
\bigbreak

The condition \eqref{e:relation} means that the function $r\to |r|^{n-1}\hat{f}(r\theta)$ is positive definite. We will use this property 
to define a more general class of intersection functions. Also, this condition allows to point out several examples of functions which are and are not
intersection functions, as follows.

\begin{example} \label{e:exponentials} Fix $\alpha,\beta >0$ and $\ell \in C(S^{n-1})$ even and strictly positive.

\begin{itemize}

    \item[(1)] For each $\theta \in S^{n-1}$ consider the function $h_{\theta}(r) = \alpha \exp(-|r|^2 \ell(\theta))$.  It can be check that $h_{\theta}$ is the Fourier transform of the non-negative function:
\[
(h_{\theta})_r^{\wedge}(t) = \alpha \sqrt{\frac{\pi}{\ell(\theta)}} e^{-\frac{1}{4\ell(\theta)} |t|^2} \geq 0.
\]
So, by Proposition~\ref{t:fourier},
\[
f(\xi) = \frac{1}{\pi}\left[|x|_2^{-n+1} \alpha e^{-|x|_2^2 \ell\left(\frac{x}{|x|_2}\right)}  \right]_x^{\wedge}(\xi)
\]
is the intersection function of the function
\[
g(t,\theta) = \alpha \sqrt{\frac{\pi}{\ell(\theta)}} e^{-\frac{1}{4 \ell(\theta)} |t|^2}
\]

\item[(2)] For each $\theta \in S^{n-1}$ consider the function $h_{\theta}(r) = \alpha \sqrt{\frac{\pi}{\beta}} e^{-\frac{1}{4\beta} |r|^2} \ell(\theta).$
Note that, as above,
\[
(h_{\theta})_r^{\wedge}(t) = \alpha \ell(\theta) e^{-\beta |t|^2} \geq 0.
\]
Again, according to Proposition~\ref{t:fourier}, 
\[
f(\xi) = \frac{1}{\pi} \left[|x|_2^{-n+1}\ell\left(\frac{x}{|x|_2}\right) \alpha \sqrt{\frac{\pi}{\beta}} e^{-\frac{1}{4\beta} |x|_2^2} \right]_x^{\wedge}(\xi) 
\]
is the intersection function of the function
\[
g(t,\theta) = \alpha \ell(\theta) e^{-\beta |t|^2}.
\]

\item[(3)] For each $\theta \in S^{n-1}$ consider the function $h_{\theta}(r) = \exp(-|r|\ell(\theta)).$
Notice that 
\[
(h_{\theta})_r^{\wedge}(t) = \frac{2\ell(\theta)}{t^2 + [\ell(\theta)]^2} \geq 0.
\]
Consequently, the function 
\[
f(\xi) = \frac{1}{\pi} \left[|x|_2^{-n+1}e^{-|x|_2 \ell\left(\frac{x}{|x|_2}\right)} \right]_x^{\wedge}(\xi)
\]
is the intersection function of 
\[
g(t,\theta) = \frac{2\ell(\theta)}{t^2 + [\ell(\theta)]^2}
\]

\item[(4)] More generally, fix $q \in (0,2]$, and for each $\theta \in S^{n-1}$, set 
\[
h_{\theta}(r) = \ell(\theta) e^{-|r|^q}.
\]
According to \cite[Lemma~2.27]{Kold1} 
\[
(h_{\theta})_r^{\wedge}(t) = \ell(\theta) \left(e^{-|r|^q}\right)_r^{\wedge}(t) := \ell(\theta) \gamma_q(t)
\]
is a positive function on $\R$. Consequently, the function
\[
f_q(\xi) = \frac{1}{\pi}\left[|x|_2^{-n+1}\ell\left(\frac{x}{|x|^2}\right)e^{-|x|_2^q} \right]_x^{\wedge}(\xi) 
\]
is the intersection function of 
\[
g_q(t,\theta) = \ell(\theta) \gamma_q(t).
\]

\end{itemize}
\end{example}

\begin{example}
To provide examples of functions which are not an intersection functions, for any $\theta \in S^{n-1}$ and $q>2$, consider functions of the form $h_{\theta}(r) = \ell(\theta)\exp(-|r|^{q}),$ where $\ell \in C(S^{n-1})$ is strictly positive. Taking the Fourier transform by $r \in \R$, we see that
\[
(h_{\theta})^{\wedge}_r(t) = \ell(\theta) (e^{-|r|^q})_r^{\wedge}
(t) := \ell(\theta) \gamma_q(t).
\]
But $\gamma_q(t)$ is not always non-negative (see \cite{Kold1}), so according to Corollary \ref{t:fourier-radon}, the function $f$ given by 
\[
f(x) = (2\pi)^{-n}\left[|x|_2^{-n+1}\ell\left( \frac{x}{|x|_2}\right) e^{-|x|_2^q}\right]^{\wedge}_{\xi}(x)
\]
fails to be an intersection function of any member of $\mathcal{L}^n.$

\end{example}

\section{Intersection Functions}  \label{s:intersection functions}

The next step is to define the class of functions which includes both intersection functions of functions and intersection bodies.  
We base our definition on the result of Corollary \ref{t:fourier-radon}, rather than use a more geometric definition in the spirit of
Lutwak's approach to intersection bodies which is based on the extension of the dual Radon transform to measures. 

\begin{definition}
\label{d:intfunction} A  non-negative, even, continuous, integrable function $f$ on $\R^n$ is called an {\it intersection function} if, for every direction $\theta \in S^{n-1}$, the function 
\[
r \in \R \mapsto |r|^{n-1} \hat{f} (r\theta)
\]
is a positive definite function on $\R$.
\end{definition}

The geometric definition now becomes a theorem, as follows. The theorem was formulated in the Introduction; we recall the statement.

\begin{reptheorem}{t:classification} An even, continuous, non-negative, and integrable function $f$ defined on $\R^n$ is an intersection function if, and only if, for every direction $\theta \in S^{n-1}$, there exists a non-negative, even, finite Borel measure $\mu_{\theta}$ on $\R$ such that 
\begin{itemize}
    \item the function 
\[
\theta \in S^{n-1} \mapsto \int_{\R} \mathcal{R}\varphi(t,\theta) d\mu_{\theta}(t)
\]
belongs to $L(S^{n-1})$ whenever $\varphi \in \mathcal{S}(\R^n)$, and 
\item 
\[
\int_{\R^n}f \varphi = \int_{S^{n-1}} \int_{\R} \mathcal{R}\varphi(t,\theta) d\mu_{\theta}(t) d\theta.
\]
holds for all $\varphi \in \mathcal{S}(\R^n).$
\end{itemize}
    
\end{reptheorem}

\begin{proof}
Begin by recalling the connection between the Radon transform and the Fourier transform: For any fixed direction $\theta \in S^{n-1}$, the function $g(t) = \mathcal{R} \hat{\varphi} (t,\theta)$ is the Fourier transform of the function $h(z) = (2\pi)^{n-1} \varphi(z\theta)$, $t,z \in \R$ whenever $\varphi$ is a test function on $\R^n$.

Assume that $f$ is an intersection function. For each $\theta \in S^{n-1}$ we are tasked with finding a finite, positive Borel measure $\mu_{\theta}$ on $\R$ for which 
\[
\int_{\R^n}f(x) \varphi(x) dx = \int_{S^{n-1}}\int_{\R} \mathcal{R}\varphi(t,\theta) d\mu_{\theta}(t)d\theta
\]
holds whenever $\varphi \in \mathcal{S}(\R^n).$

Since $f$ is an intersection function on $\R^n$, for every direction $\theta \in S^{n-1}$, the function $h_{\theta}(r) = |r|^{n-1} \hat{f}(r\theta)$ is a positive definite function on $\R$. Therefore, by Bochner's theorem, for each $\theta \in S^{n-1}$, there exists a finite, positive Borel measure $\nu_{\theta}$ on $\R$ such that the Fourier transform of $\nu_{\theta}$ is equal to $h_{\theta}$. 
Notice that, by applying Parseval's identity on $\R^n$ and then again on $\R$, we have 
\begin{align*}
\langle f, \varphi \rangle &= (2\pi)^{n} \langle \hat{f}, \hat{\varphi} \rangle =\frac{(2\pi)^n}{2} \int_{S^{n-1}} \int_{\R}|r|^{n-1}\hat{f}(r\theta)\hat{\varphi}(r\theta) dr d\theta\\
&= \frac{(2\pi)^n}{2} \int_{S^{n-1}}\int_{\R}(|\cdot|^{n-1}\hat{f}(\cdot\theta))_r^{\wedge}(z)(\hat{\varphi}(\cdot \theta))_r^{\wedge}(z) dz d\theta\\
&=\frac{(2\pi)^n}{2} \int_{S^{n-1}}\int_{\R} \mathcal{R}\varphi(s,\theta) d\nu_{\theta}(s)d\theta
\end{align*}
whenever $\varphi \in \mathcal{S}(\R^n)$ is even, which is exactly the condition \eqref{e:intfunrelation}. 
 
Conversely, assume that the condition \eqref{e:intfunrelation} holds. For every fixed direction $\theta$, by Bochner's theorem, the Fourier transform of the measure $\mu_{\theta}$ is a continuous, positive definite function $f_{\theta}$ defined on $\R$. Consequently, for any even test function $\varphi \in \mathcal{S}(\R^n)$, applying Parseval's identity to the integral by $dt$, we have that 
\begin{align*}
\langle \hat{f}, \varphi \rangle &= \langle f, \hat{\varphi} \rangle= \int_{S^{n-1}} \int_{\R} \mathcal{R}\hat{\varphi}(t,\theta) d\mu_{\theta}(t) d\theta\\
&=\int_{S^{n-1}} \langle \mathcal{R} \hat{\varphi} (\cdot,\theta), \mu_{\theta}(\cdot) \rangle d\theta = \int_{S^{n-1}} \langle [\mathcal{R}\hat{\varphi}(\cdot, \theta)]_t^{\wedge}, [\mu_{\theta}]_t^{\wedge} \rangle d\theta\\
&= (2\pi)^{n-1}\int_{S^{n-1}}\int_{\R} \varphi(z\theta) f_{\theta}(z)dz d\theta\\
&= 2(2\pi)^{n-1} \int_{S^{n-1}}\int_{0}^{\infty} \frac{|r|^{n-1}}{|r|^{n-1}} \varphi(z\theta) f_{\theta}(z) dz d\theta\\
&=2(2\pi)^{n-1} \int_{\R^n}|x|_2^{-n+1} \varphi(x) f_{\frac{x}{|x|_2}}(|x|_2)dx\\
&= 2(2\pi)^{n-1} \langle |x|_2^{-n+1} f_{\frac{x}{|x|_2}}(|x|_2), \varphi \rangle.
\end{align*} 
So it must be the case that 
\[
\hat{f}(x) = |x|_2^{-n+1}f_{\frac{x}{|x|_2}}(|x|_2)
\]
as distributions.

Hence, for any fixed direction $\theta \in S^{n-1}$, using the positive definiteness of the function $f_{\theta}$, one has 
\begin{align*}
\langle [|\cdot|^{n-1} \hat{f}(\cdot\theta)]_r^\wedge, \psi \rangle &= \langle |\cdot|^{n-1}\hat{f}(\cdot \theta), \hat{\psi} \rangle\\
&= \int_{\R} f_{\theta}(z) (\psi)^{\wedge}_t(z) dz\\
&= \int_{\R} (f_{\theta})_z^{\wedge}(s) \psi(s) ds \geq 0,
\end{align*}
whenever $\psi \in \mathcal{S}(\R)$ is even and non-negative.  Therefore, the function $f$ is an intersection function. 

\end{proof}

In the following subsections we will examine some examples of intersection functions. In particular, we will see that the class of intersection functions contains the class of intersection bodies of star bodies. 

\subsection{The spherical Radon transform}

Let $\ell \in C(S^{n-1})$ be continuous and strictly positive. Given $\e >0$, consider the function
\[
g_{\e}(t,\theta) = \ell(\theta) \frac{1}{\sqrt{\pi \e}} e^{-\e^2|t|^2}. 
\]
As we saw in Example~\ref{e:exponentials}, for each $\e>0$, the intersection function $f_{\e}$ of $g_{\e}$ is 
\[
f_{\e}(\xi) = \frac{1}{\pi} \left[|x|_2^{-n+1}\ell\left(\frac{x}{|x|_2}\right)  \sqrt{\frac{\pi}{\e}} e^{-\frac{1}{4\e^2} |x|_2^2} \right]_x^{\wedge}(\xi).
\]
By the definition of an intersection function of $g$, for every even $\varphi \in S(\R^n)$ and any fixed $\e>0$, 
\begin{align*}
&\int_{\R^n}\frac{1}{\pi} \left[|x|_2^{-n+1}\ell\left(\frac{x}{|x|_2}\right)  \sqrt{\frac{\pi}{\e}} e^{-\frac{1}{4\e^2} |x|_2^2} \right]_x^{\wedge}(\xi) \varphi(\xi) d\xi = \int_{\R^n}f_{\e}(\xi) \varphi(\xi)d\xi\\
&=\int_{S^{n-1}}\int_{\R} \mathcal{R}\varphi(t,\theta) g_{\e}(t,\theta)dt d\theta \\
&= \int_{S^{n-1}} \int_{\R} \mathcal{R}\varphi(t,\theta) \ell(\theta) \frac{1}{\sqrt{\pi \e}} e^{-\e^2|t|^2} dtd\theta.\\
\end{align*}
Sending $\e \to 0$, the left-hand side of the above equality tends to 
\begin{equation*}
\begin{split}
&\frac{1}{\pi}\int_{\R^n} \left[|x|_2^{-n+1} \ell\left(\frac{x}{|x|_2}\right)(\delta_0(t))_t^{\wedge}(|x|_2) \right]_x^{\wedge}(\xi) \varphi(\xi) d \xi\\
&= \frac{1}{\pi}\int_{\R^n} \left[|x|_2^{-n+1} \ell\left(\frac{x}{|x|_2} \right) \right]_x^{\wedge}(\xi) \varphi(\xi) d \xi
\end{split}
\end{equation*}
whenever $\varphi \in S(\R^n)$ is even, where $\delta_0$ is the delta function. Here we have used the fact that $\hat{\delta} \equiv 1$. Similarly, $\e \to 0$ the right-hard side tends to 
\begin{align*}
\int_{S^{n-1}} \mathcal{R}\varphi(0,\theta) \ell(\theta) d\theta = \int_{S^{n-1}}R\varphi(\theta) \ell(\theta)d\theta,
\end{align*}
where $R \colon C(S^{n-1}) \to C(S^{n-1})$ denotes the spherical Radon transform.  

Consequently, we have shown that 
\begin{equation}\label{e:spherical condition} 
\frac{1}{\pi}\int_{\R^n} \left[|x|_2^{-n+1} \ell\left(\frac{x}{|x|_2} \right) \right]_x^{\wedge}(\xi) \varphi(\xi) d \xi = \int_{S^{n-1}}R\varphi(\theta) \ell(\theta)d\theta  
\end{equation}
whenever $\varphi \in \mathcal{S}(\R^n)$ is even.
From Theorem~\ref{t:classification} paired with Bochner's theorem, limits of intersection functions are themselves intersection functions, so it follows that 
\[
f(\xi) = \frac{1}{\pi} \left[|x|_2^{-n+1}\ell \left(\frac{x}{|x|_2}\right) \right]^{\wedge}_x(\xi)
\]
is an intersection function. In particular, $f$ is a continuous function on the sphere extended to a homogeneous function of the order $-1$
on $\R^n\setminus \{0\},$ which recover \cite[Lemma~3.7]{Kold1}: For every $\theta \in S^{n-1}$, 
\[
\frac{1}{\pi}\left[|x|_2^{-n+1}\ell\left(\frac{x}{|x|_2}\right) \right]_x^{\wedge}(\theta) = R\ell(\theta).
\]

Next, we can rewrite equality \eqref{e:spherical condition} to get 
\begin{align*}
\int f(x) \varphi(x) dx &= \int_{S^{n-1}} R\varphi(\theta) \ell(\theta) d\theta\\
&=\int_{S^{n-1}} \left(\int_{\langle x, \theta \rangle = 0} \varphi(x) dx\right) \ell(\theta) d\theta\\
&= \int_{S^{n-1}} \left(\int_{S^{n-1} \cap \theta^{\perp}} \int_0^{\infty} s^{n-2}\varphi(s\omega) ds d\omega \right)\ell(\theta) d\theta.
\end{align*}

If we denote by $h(\theta)=\int_0^\infty r^{n-2}\varphi(r\theta) dr,$ then we recover the well-known self-duality
property of the spherical Radon transform: For any $\ell,h\in C(S^{n-1})$
$$\int_{S^{n-1}} R\ell(\theta) h(\theta) d\theta = \int_{S^{n-1}} Rh(\theta) \ell(\theta) d\theta.$$

\subsection{Intersection bodies.} 
In the previous example, set $\ell(\theta) = \|\theta\|_L^{-n+1},$ where $L$ is an origin-symmetric star body in $\R^n$. Then we recover the Fourier formula for the volume of a section, \cite[Th.3.8]{Kold1}:
$$f(x)=(n-1)|x|_2^{-1} \left|L\cap\left(\frac x{|x|_2}\right)^\bot\right|=\frac{1}{\pi}(\|\cdot\|_L^{-n+1})^\wedge(x).$$
In fact, we have shown that the concept of intersection function as described in Definition~\ref{d:intfunction} extends the notion of intersection  bodies.

Moreover, we have $f(x)=(n-1)\|x\|_{IL}^{-1},$ so we recover the result of \cite[p.72]{Kold1}: For every $\xi\in S^{n-1}$
$$\left(\|x\|_{IL}^{-1}\right)^\wedge(\xi) = \frac{(2\pi)^n}{\pi(n-1)}\|\xi\|_L^{-n+1}.$$
In particular, an origin-symmetric star body $K$ is an intersection body of a star body if, and only if,
the Fourier transform of $\|\cdot\|_K^{-1}$ is a $(-n+1)$-homogeneous function on $\R^n$ whose restriction
to the sphere is continuous and strictly positive, cf \cite[Th.4.1]{Kold1}.

\section{The case of the Radon transform}\label{s:Radonaffirmative}

In this section, we prove Theorem~\ref{t:L2BP} and Theorem~\ref{t:RadonComparisonNegative}. 

\begin{reptheorem}{t:L2BP} Let $p >0$ and consider a pair of continuous, non-negative even functions $\varphi,\psi \in L^1(\R^n) \cap L^p(\R^n)$ satisfying the condition 
\[
\mathcal{R}\varphi(t,\theta) \leq \mathcal{R}\psi(t,\theta) \quad \text{for all } (t,\theta) \in \R \times S^{n-1}.
\] Then:
\begin{itemize}
    \item[(a)] if $p >1$ and $\varphi^{p-1}$ is an intersection function, then $\|\varphi\|_{L^p(\R^n)} \leq \|\psi\|_{L^p(\R^n)}$, and 
    \item[(b)] if $0 < p <1$ and $\psi^{p-1}$ is an intersection function, then $\|\varphi\|_{L^p(\R^n)} \leq \|\psi\|_{L^p(\R^n)}$.
\end{itemize}
\end{reptheorem}

\begin{proof} Without loss of generality, we may assume that $\varphi, \psi \in \mathcal{S}(\R^n)$. 

We begin the proof of (a). Since $\varphi^{p-1}$ is an intersection function, by Theorem~\ref{t:classification}, for each $\theta \in S^{n-1}$ there exists a non-negative, even, finite Borel measure $\mu_{\theta}$ on $\R$ such that the function
\[
\alpha_{\theta}:= \int_{\R}\mathcal{R}\alpha(t,\theta) d\mu_{\theta}(t)
\]
is integrable on $S^{n-1}$ for any $\alpha \in \mathcal{S}(\R^n)$. Integrating both sides of the assumption \eqref{e:LpBPassump} over $\R$ with respect to the measure $\mu_{\theta}$, we then have the inequality 
\[
\varphi_{\theta} \leq \psi_{\theta} \quad \text{for all } \theta \in S^{n-1}.
\]
Integrating the above inequality over $S^{n-1}$ and applying the identity \eqref{e:intfunrelation} of Theorem~\ref{t:classification}, we obtain
\begin{align*}
\int_{\R^n}\varphi(x)^p dx &= \int_{\R \times S^{n-1}}\mathcal{R} \varphi(t\theta) d\mu_{\theta}(t)d\theta\\
&= \int_{S^{n-1}} \varphi_{\theta} d\theta \leq \int_{S^{n-1}} \psi_{\theta} d\theta\\
&=\int_{\R \times S^{n-1}}\mathcal{R} \psi(t\theta) d\mu_{\theta}(t)d\theta = \int_{\R^n} \varphi(x)^{p-1} \psi(x) dx\\
&\leq \left(\int_{\R^n} \varphi(x)^{p}dx \right)^{\frac{p-1}{p}} \left(\int_{\R^n} \psi(x)^{p}dx\right)^{\frac{1}{p}},
\end{align*}
where in the last line we applied H\"older's inequality.

The proof of part (b) is the same as (a) with a minor adjustment akin to the proof of Theorem~\ref{t:sphericalAffirmativeCase}(b). 
\end{proof}

% One can rewrite the latter theorem in the form of an isomorphic result.
%\begin{corollary}\label{t:Radonisomorphic}
%Let  $p > 1$, and consider a pair of functions $\varphi, \psi \in L^1(\R^n) \cap L^{p}(\R^n)$ that are non-negative, even, and continuous, and suppose that $\mathcal{R}\psi(t,\theta) >0$ for all $(t,\theta)$.Then, if $\varphi^{p-1}$ is an intersection function, one has 
%\[
%\|\varphi\|_{L^{p}(\R^n)} \leq \sup_{(t,\theta) \in \R \times S^{n-1}}\left(\frac{\mathcal{R} \varphi(t,\theta)}{\mathcal{R} \psi(t,\theta)}\right) \|\psi\|_{L^{p}(\R^n)}.
%\]
%\end{corollary}
%To see this, just observe that, for every $(t,\theta) \in \R \times S^{n-1}$, one has 
%\[
%\mathcal{R}\varphi(t,\theta) = \frac{\mathcal{R}\varphi(t,\theta)}{\mathcal{R}\psi(t,\theta)}\mathcal{R}\psi(t,\theta) \leq \sup_{(t,\theta) \in \R \times S^{n-1}} \left(\frac{\mathcal{R}\varphi(t,\theta)}{\mathcal{R}\psi(t,\theta)}\right) \mathcal{R}\psi(t,\theta),
%\]
%and then apply Theorem~\ref{t:L2BP}.
\smallbreak
Next, we treat the second part of Problem~\ref{p:fBP}, Theorem~\ref{t:RadonComparisonNegative}, which we restate here. 

\begin{reptheorem}{t:RadonComparisonNegative} The following hold: 
\begin{itemize}
    \item[(a)] Fix $p >1$ and let $\psi \in \mathcal{S}(\R^n)$ be non-negative and even. If $\psi^{p-1}$ is not an intersection function, then there exists an even, non-negative  $\varphi \in \mathcal{S}(\R^n)$ such that 
\[
\mathcal{R} \varphi(t,\theta) \leq \mathcal{R} \psi(t,\theta) \quad \text{for all } (t,\theta) \in \R \times S^{n-1},
\]
but with $\|\psi\|_{L^{p}(\R^n)} < \|\varphi\|_{L^{p}(\R^n)}$.
\item[(b)] Fix $0 < p <1$ and let $\varphi \in \mathcal{S}(\R^n)$ be non-negative and even. If $\varphi^{p-1}$ is not an intersection function, then there exists a non-negative, even $\psi \in \mathcal{S}(\R^n)$ such that $\mathcal{R}\varphi \leq \mathcal{R}\psi$, but with $\|\psi\|_{L^{p}(\R^n)} < \|\varphi\|_{L^{p}(\R^n)}$.
\end{itemize} 
\end{reptheorem}

\begin{proof}
We will present the proof of part (a). For brevity, set $\psi_{\theta}(t)=|t|^{n-1}\widehat{\psi^{p-1}}(t\theta)$ whenever $\theta \in S^{n-1}$. To begin, we will show that there is a symmetric set $\Gamma \subset S^{n-1}$ of positive $S^{n-1}$ measure such that $\psi_{\omega}$ is not a positive definite function of $t \in \R$ whenever $\omega \in \Gamma$.

Since $\psi \in \mathcal{S}(\R^n)$, we get that $\widehat{\psi^{p-1}} \in \mathcal{S}(\R^n)$. Using the fact that $\psi^{p-1}$ is not an intersection function, there exists a direction $\nu \in S^{n-1}$ such that $\psi_{\nu}$ is not positive definite on the line $\R\nu:=\{t\nu \colon t \in \R\}$. In particular, since $\psi_{\nu}$ is not positive definite on $\R\nu$, there exists some non-empty, open symmetric set $I_{\nu} \subset \R \nu$, and a non-negative, even $f \in \mathcal{S}(\R)$ such that 
\[
\langle \psi_{\nu}, \widehat{f} \rangle = \int_{\R} \psi_{\nu}(t) \widehat{f}(t)dt = -\delta <0. 
\]
Our goal is to show that there is a  small neighborhood $\Omega$ of $\nu$ in $S^{n-1}$ such that
$$|\langle \psi_{\nu} - \psi_{\omega}, \widehat{f} \rangle| \le \frac{\delta}{2},$$
for all $\omega \in \Omega$. From the Cauchy-Schwartz inequality:
$$
|\langle \psi_{\nu} - \psi_{\omega}, \widehat{f} \rangle| \le  \|\psi_{\nu} - \psi_{\omega}\|_{L^2(\R)} \|f\|_{L^2(\R)}.
$$

Since $\|f\|_{L^2(\R)}$ is bounded, it is enough to show that $\|\psi_{\nu} - \psi_{\omega}\|_{L^2(\R)}$ is small whenever $\omega,\nu \in S^{n-1}$ are sufficiently close in norm. Using the fact that $\widehat{\psi^{p-1}} \in \mathcal{S}(\R^n)$, it is a locally Lipschitz function. Moreover, there exists $m \in \mathbb{N}$ and $M>0$ such that 
\[
|\widehat{\psi^{p-1}}(x)| \leq \frac{1}{|x|^m} \quad \text{for all } |x|_2 > M,\quad \text{and } 
\int_R^\infty t^{2(n-1)}|\widehat{\psi^{p-1}}(t\omega)|^2dt \le \frac{\delta^2}{16\|f\|_{L^2(\R)}^2}\]
holds for all $\omega \in S^{n-1}$.  

 Let $C(\psi)= 2M^{2n+2} L_{(\widehat{\psi^{p-1}})}^2$, where $L_{\widehat{\psi^{p-1}}}$ is the Lipschitz constant of $\widehat{\psi^{p-1}}$ on $[0,R]$. 
 Now, if $\omega \in S^{n-1}$ satisfies $|\nu -\omega|_2 <\sqrt{\frac{\delta^2}{8 C(\psi) \|f\|_{L^2(\R)}}}$, then 
\begin{align*}
&\|\psi_{\nu}- \psi_{\omega}\|_{L^2(\R)}^2 =\int_{\R}|\psi_{\nu}(r) - \psi_{\omega}(t)|^2 dt\\
&=2\int_0^M t^{2(n-1)}|(\widehat{\psi^{p-1}})(t\nu)-(\widehat{\psi^{p-1}})(t\omega)|^2 dt+ 2 \int_M^{\infty}t^{2(n-1)}|(\widehat{\psi^{p-1}})(t\nu)-(\widehat{\psi^p})(t\omega)|^2 dt \\
&\leq 2\int_0^Mt^{2(n-1)}|(\widehat{\psi^{p-1}})(t\nu)-(\widehat{\psi^{p-1}})(t\omega)|^2 dt +\frac{\delta^2}{4 \|f\|_{L^2(\R)}}\\
&\leq 2M^{2n+2} L_{(\widehat{\psi^{p-1}})}^2|\nu - \omega|_2^2+ \frac{\delta}{ 4\|f\|_{L^2(\R)}} \leq \frac{\delta^2}{2\|f\|_{L^2(\R)}}.
\end{align*}
Therefore, given $\omega \in S^{n-1}$ with $|\nu -\omega|_2 <\sqrt{\frac{\delta^2}{8 C(\psi) \|f\|_{L^2(\R)}}}$, the Cauchy-Schwartz inequality implies
\[
\langle \psi_{\omega}, \widehat{f} \rangle = \langle \psi_{\omega} - \psi_{\nu}, \widehat{f} \rangle + \langle \psi_{\nu}, \widehat{f} \rangle \leq \|f\|_{L^2} \|\psi_{\omega} -\psi_{\nu}\|_{L^2(\R)} - \delta < \frac{\delta}{2} - \delta< 0. 
\]
It follows that the function $\psi_{\omega}$ is not positive definite on $\R$ whenever $\omega \in S^{n-1}$ is sufficiently close to $\nu$. Set
\[
\Omega := \left\{\omega \in S^{n-1} \colon |\nu - \omega|_2 < \sqrt{\frac{\delta^2}{8 C(\psi) \|f\|_{L^2(\R)}}} \right\}. 
\]
Denote by $\sigma$ the uniform measure on $S^{n-1}$. Then $\sigma(\Gamma) >0$ and that $\Omega$ is symmetric. Moreover, observe that $\psi_{\omega}$ is not positive definite on $\R$ whenever $\omega \in \Gamma$. 

We extend $\Omega$ to $\R^n$ by letting 
\[
\widetilde{\Omega} := \bigcup_{\omega \in \Gamma} \R\omega.
\]
which is symmetric in $\R^n$.  Since $\psi_{\omega}$ is not positive definite on the line $\R\omega$, and $\psi_{\omega} \in \mathcal{S}(\R)$, it follows that there exists a non-empty symmetric open set $\Lambda_{\omega} \subset \R \omega$ on which $\widehat{\psi_{\omega}}<0$.  Finally, we let 
\[
\Lambda := \bigcup_{\omega \in \Gamma} \Lambda_{\omega}.
\]

Let $h \in \mathcal{S}(\R^n)$ be non-negative, even, and such that $h > 0$ only on the set $\Lambda$. Consider the function $\varphi$ defined by 
\[
\varphi(x) = \psi(x) -\eta h(x),
\]
with $\eta$ sufficiently small so that $\varphi \geq 0$ on $\R^n.$ Notice that, for any $(t,\theta) \in \R \times S^{n-1}$, we obtain 
\[
\mathcal{R}\varphi(t,\theta) = \mathcal{R}\psi(t,\theta) - \eta \mathcal{R}h(t,\theta) \leq \mathcal{R}\psi(t,\theta). 
\]

Now, for each fixed direction in the sphere $\theta \in S^{n-1}$, we have 
\begin{align*}
\int_{\R}\widehat{\psi_{\theta}}(t)\mathcal{R}\varphi(t,\theta) dt &= \int_{\R} \widehat{\psi_{\theta}}(t)\mathcal{R}\psi(t,\theta)-\eta\int_{\R}\widehat{\psi_{\theta}}(t)\mathcal{R}h(t,\theta)\\
&\geq \int_{\R} \widehat{\psi_{\theta}}(t)\mathcal{R}\psi(t,\theta),
\end{align*}
However, if $\theta \in \Gamma$, then the last inequality is strict, since  $h > 0$ only on the admissible set $\Lambda$. For each fixed $\theta \in S^{n-1}$, applying the Parseval identity to both sides of the above inequality, we obtain 
\[
\int_{\R}|s|^{n-1} \widehat{\psi^{p-1}}(s\theta) \widehat{\varphi}(s\theta)ds \geq \int_{\R}|s|^{n-1} \widehat{\psi^{p-1}}(s\theta) \widehat{\psi}(s\theta)dr.
\]
Integrating the above inequality over $S^{n-1}$, we obtain 
\begin{align*}
&\int_{S^{n-1}} \int_{\R}|r|^{n-1} \widehat{\psi^{p-1}}(r\theta) \widehat{\varphi}(r\theta)dr d\theta\\
&=\int_{S^{n-1} \setminus \Gamma} \int_{\R}|s|^{n-1} \widehat{\psi^{p-1}}(r\theta) \widehat{\varphi}(s\theta)dr d\theta + \int_{\Gamma} \int_{\R}|s|^{n-1} \widehat{\psi^{p-1}}(s\theta) \widehat{\varphi}(s\theta)ds d\theta\\
&>\int_{S^{n-1} \setminus \Gamma} \int_{\R}|s|^{n-1} \widehat{\psi^{p-1}}(s\theta) \widehat{\psi}(s\theta)ds d\theta + \int_{\Gamma} \int_{\R}|s|^{n-1} \widehat{\psi^{p-1}}(s\theta) \widehat{\psi}(s\theta)ds d\theta\\
&=\int_{S^{n-1}} \int_{\R}|s|^{n-1} \widehat{\psi^{p-1}}(s\theta) \widehat{\psi}(s\theta)ds d\theta,
\end{align*}
where, in the second to last line, we used the fact that $\sigma(\Gamma) > 0$. 
Using symmetry of the functions $\psi$ and $\varphi$, and integrating in polar coordinates, we deduce the inequality 
\[
\int_{\R^n}\widehat{\psi^{p-1}}(x)\widehat{\varphi}(x)dx > \int_{\R^n}\widehat{\psi^{p-1}}(x)\widehat{\psi}(x)dx.
\]
Applying Parseval's identity to the above inequality followed by H\"older's inequality, we complete the proof. 

The proof of (b) follows from combining the above proof with the ideas in proof of Theorem~\ref{t:sphericalNegativeCase}b.
\end{proof}

%%%%%%%%%%%%%%%%%%%%%%%%%%%%%%%%%%
\section*{Acknowledgments}

M.R. would like to thank the Department of Mathematics at the University of Missouri. M.R. and A.Z would like to thank Sorbonne University and LAMA at Universit\'e Gustave Eiffel for wonderful and productive stays during which a significant part of this manuscript was produced.  M.R. would also like to thank Effrosyni Chasioti for many helpful discussions about the content of the article. We woule like to thank Dimtry Ryabogin for suggesting we consider the case of  $p \in (0,1)$ in Problems~\ref{p:SphericalComparison} and \ref{p:fBP}.

%%%%%%%%%%%%%%%%%%%%%%%%%%%%%%%%%%%%%%%%%

\vspace{3mm}

\noindent {\sc Department of Mathematics, University of Missouri, Columbia. MO, USA.}
\noindent {\it E-mail address:} {\tt koldobskiya@missouri.edu}

\vspace{2mm}

\noindent {\sc  Institute for Computational and Experimental Research in Mathematics, Brown University, Providence, RI USA.\\}
\noindent {\it E-mail address:} {\tt michael\_roysdon@brown.edu}

\vspace{2mm}

\noindent {\sc Department of Mathematical Sciences, Kent State University, Kent, Ohio USA.}
\noindent {\it E-mail address:} {\tt zvavitch@math.kent.edu}

\


\begin{thebibliography}{99}
\addcontentsline{toc}{section}{References}
\setlength{\itemsep}{1pt}


\bibitem{AGM1}  S.~Artstein-Avidan, A.~Giannopoulos, and  V.D.~Milman, {\it {A}symptotic {G}eometric {A}nalysis. Part I}, Mathematical Surveys and Monographs, 202. American Mathematical Society, Providence,
RI, 2015. xx+451 pp.

\bibitem{AGM2} S.~Artstein-Avidan, A.~Giannopoulos, and V.D~Milman, {\it {A}symptotic {G}eometric {A}nalysis. Part II}, Mathematical Surveys and Monographs., 262. American Mathematical Society, Providence,
RI, 2021. xxxvii+645 pp.

\bibitem{Bourgain0} J.~Bourgain, {\it On high-dimensional maximal functions associated to convex bodies}, Amer. J.Math. 108 (1986): 1467–1476.

\bibitem{Bourgain2} J.~Bourgain, {\it Geometry of Banach spaces and harmonic analysis}, Proceedings of the International
Congress of Mathematicians (Berkeley, CA, 1986), Amer. Math. Soc., Providence, RI, 1987, 871–878.

\bibitem{BP} H.~Busemann and C.M.~Petty, {\it Problems on convex bodies}, Math. Scand. 4 (1956), 88–94.

\bibitem{CGL} G.~Chasapis, A.~Giannopoulos and D.~Liakopoulos, {\it Estimates for measures of lower dimensional
sections of convex bodies}, Adv. Math. 306 (2017), 880–904.

\bibitem{Chen} Y.~Chen, {\it An almost constant lower bound of the isoperimetric coefficient in the KLS conjecture},
Geom. Funct. Anal. 31 (2021), 34–61.

\bibitem{Christ} M.~Christ, {\it Estimates for the k-plane transform}, Indiana Univ. Math. J. 33 (1984), no. 6, 891–910.

\bibitem{Christ2} M.~Christ, {\it Extremizers of a Radon transform inequality}, Advances in analysis: the legacy of Elias M. Stein,
Princeton Math. Ser., 50, Princeton Univ. Press, Princeton, NJ, 2014, 84-107. 

\bibitem{Gardner} R.J.~Gardner,  {\it Geometric Tomography}. Second edition, Cambridge University Press, Cambridge,
2006, xxii+492 pp.

\bibitem{GKS} R.J.~Gardner, A.~Koldobsky and Th.~Schlumprecht, {\it An analytic solution to the Busemann-
Petty problem on sections of convex bodies}, Annals of Math. 149 (1999), 691–703.
\bibitem{GS} I.M.~Gelfand and G.E.~Shilov,  {\it Generalized Functions, vol.1 Properties and Operations}, Academic Press, New York-London 1964 xviii+423 pp. 

\bibitem{GV} I.M.~Gelfand and N.Ya.~Vilenkim,  {\it Generalized Functions vol. 4. Applications of Harmonic Analysis}, Academic Press, New York-London 1966 xvii+449 pp.

\bibitem{GK} A.~Giannopoulos and A.~Koldobsky, {\it Variants of the Busemann-Petty problem and of the Shephard
problem}, Int. Math. Res. Not. (IMRN), 2017, no. 3, 921–943.

\bibitem{GKZ} A.`Giannopoulos, A.~Koldobsky and A.~Zvavitch, {\it Inequalities for the Radon transform on convex
sets}, Int. Math. Res. Not. (IMRN), 2022, no. 18, 13984–14007.

\bibitem{GKZ2} A.Giannopoulos, A.~Koldobsky, and A.~Zvavitch, {\it Inequalities for sections and projections of
convex bodies.}  2023. https://doi.org/10.1515/9783110775389-006.

\bibitem{GLW} P.~Goodey, E.~Lutwak, and W.~Weil, {\it Functional analytic characterizations of classes of convex bodies},Math. Z. 222 (1996), no. 3, 363–381.

\bibitem{GK} W~.Gregory and A.~Koldobsky, {\it Inequalities for the derivatives of the Radon transform on
convex bodies}, Israel J. Math. 246 (2021), no. 1, 261–280.

\bibitem{Grinberg} E. L..~Grinberg,  {\it Isoperimetric inequalities and identities for k-dimensional cross-sections
of convex bodies.} Math. Ann. 291 (1991): 75–86.

\bibitem{Groemer} H.~Groemer,
	{\it {G}eometric {A}pplications of {F}ourier {S}eries and {S}pherical {H}armonics}, Encyclopedia of Mathematics and its Applications, 61. Cambridge University Press, Cambridge, 1996. xii+329 pp.

\bibitem{Helgason}  S.~Helgason, {\it The Radon transform}, Second edition. Progress in Mathematics, 5. Birkhäuser Boston, Inc., Boston, MA, 1999. xiv+188 pp.

\bibitem{Helgason2} S.~Helgason, {\it Integral Geometry and Radon Transforms}, Springer, New York, 2011. xiv+301

\bibitem{JLV} A.~Jambulapati, Y.T.~Lee. and S.~Vempala, {\it A slightly improved bound for the KLS constant}, arxiv:2208.11644v2.

\bibitem{John} F.~John, {\it Extremum problems with inequalities as subsidiary conditions}, Courant Anniversary
Volume, Interscience, New York (1948), 187–204.

\bibitem{Klartag0} B.~Klartag, {\it An isomorphic version of the slicing problem}, J. Funct. Anal. 218 (2005), 372–394.

\bibitem{Klartag} B.~Klartag, {\it On convex perturbations with a bounded isotropic constant}, Geom. Funct. Anal.
16 (2006), 1274–1290.

\bibitem{Klartag2} B.~Klartag, {\it Logarithmic bounds for isoperimetry and slices of convex sets}, arXiv:2303.14938 

\bibitem{KK} B.~Klartag and A~ Koldobsky, {\it An example related to the slicing inequality for general measures},
J. Funct. Anal. 274 (2018), 2089–2112.

\bibitem{KLeh} B.~Klartag and J.~Lehec, {\it Bourgain’s slicing problem and KLS isoperimetry up to polylog}, Geom.
Funct. Anal. 32 (2022), no. 5, 1134–1159.

\bibitem{KL} B. Klartag and G. V. Livshyts, {\it The lower bound for Koldobsky’s slicing inequality via random
rounding}, Geometric aspects of functional analysis. Vol. II, 43–63, Lecture Notes in Math. 2266,
Springer, Cham, 2020.

\bibitem{Kold1} A.~Koldobsky, {\it Fourier Analysis in Convex Geometry}, Mathematical Surveys and Monographs,
116. American Mathematical Society, Providence, RI, 2005. vi+170 pp.

\bibitem{Kold2} A.~Koldobsky, {\it A hyperplane inequality for measures of convex bodies in $\R^n$, $n\leq 4$}, Discrete
Comput. Geom. 47 (2012), 538–547.

\bibitem{Kold3} A.~Koldobsky, {\it A $\sqrt{n}$ estimate for measures of hyperplane sections of convex bodies}, Adv. Math.
254 (2014), 33–40.

\bibitem{Kold4} A.~Koldobsky, {\it Slicing inequalities for measures of convex bodies}, Adv. Math. 283 (2015), 473–
488.

\bibitem{Kold5} A.~Koldobsky, {\it Intersection bodies, positive definite distributions and the Busemann-Petty problem},
Amer. J. Math. 120 (1998), 827–840.

\bibitem{Kold6} A.~Koldobsky, {\it Isomorphic Busemann-Petty problem for sections of proportional dimensions}, Adv. in Appl. Math. 71 (2015), 138–145.

\bibitem{KP} A.~Koldobsky and A.~Pajor, {\it A remark on measures of sections of Lp-balls}, Geometric aspects
of functional analysis, 213–220, Lecture Notes in Math., 2169, Springer, Cham, 2017.

\bibitem{KPZ} A.~Koldobsky, G.~Paouris and A.~Zvavitch, {\it Measure comparison and distance inequalities for
convex bodies}, Indiana Univ. Math. J. 71 (2022), no. 1, 391–407.

\bibitem{KZ} A.~Koldobsky and A.~Zvavitch, {\it An isomorphic version of the Busemann-Petty problem for
arbitrary measures}, Geom. Dedicata 174 (2015), 261–277.

\bibitem{Lutwak} E.~Lutwak, {\it Intersection bodies and dual mixed volumes}, Adv. Math. 71 (1988), 232–261.

\bibitem{Lutwak2} E.~Lutwak, {\it Selected affine isoperimetric inequalities}, Handbook of Convex Geometry, Vol. A, B, North-Holland, Amsterdam, 1993, 151–176.

\bibitem{MP} V.D.~Milman and A.~Pajor, {\it Isotropic position and inertia ellipsoids and zonoids of the unit
ball of a normed $n$-dimensional space}, in: Geometric Aspects of Functional Analysis, ed. by
J. Lindenstrauss and V. D. Milman, Lecture Notes in Mathematics 1376, Springer, Heidelberg,
1989, pp. 64–104.

\bibitem{MPF} D.S~Mitrinovi\'c, J.E.~Pečarić, and A.M.~Fink, {\it Classical and new inequalities in analysis}. Mathematics and its Applications (East European Series), 61. Kluwer Academic Publishers Group, Dordrecht, 1993. xviii+740 pp

\bibitem{OS} D.M.~Oberlin and E.M.~Stein, {\it Mapping properties of the Radon transform}, Indiana Univ. Math. J. 31 (1982), no. 5, 641–650. 

\bibitem{Rubin} B.~Rubin, {\it  Norm estimates for k-plane transforms and geometric inequalities},
Adv. Math. 349 (2019), 29–55. 

\bibitem{Rubin1} B.~Rubin, {\it Introduction to Radon Transforms. With Elements of Fractional Calculus and Harmonic Analysis}, Encyclopedia of Mathematics and its Applications, 160. Cambridge University Press, New York, 2015. xvii+576 pp.

\bibitem{Rudin} W.~Rudin, {\it Functional Analysis}. Second edition. International Series in Pure and Applied Mathematics. McGraw-Hill, Inc., New York, 1991. xviii+424 pp.

\bibitem{Schneider} R. Schneider, {\it Convex Bodies: The Brunn-Minkowski Theory}, Second expanded edition. Encyclopedia
of Mathematics and Its Applications 151, Cambridge University Press, Cambridge,
2014.

\bibitem{Zhang} G.~Zhang, {\it Sections of convex bodies}, Amer. J. Math. 118 (1996), 319–340.

\bibitem{Zvavitch1} A.~Zvavitch, {\it Gaussian measure of sections of convex bodies}, Adv. Math. 188 (2004), no. 1, 124–136. 

\bibitem{Zvavitch2} A.~Zvavitch, {\it The Busemann-Petty problem for arbitrary measures}, Math. Ann. 331 (2005), 867–887.


\end{thebibliography}
\end{document}